\documentclass{amsart}%
\usepackage{amssymb}
\usepackage{amsmath}
\usepackage{amsfonts}
\usepackage[margin=1.5in]{geometry}
\usepackage[all,cmtip]{xy}
\usepackage{url}
\usepackage{graphicx}%
\theoremstyle{plain}
\newtheorem*{acknowledgement}{Acknowledgement}

\newtheorem{corollary}{Corollary}

\newtheorem{definition}{Definition}

\newtheorem{lemma}{Lemma}

\newtheorem{proposition}{Proposition}
\theoremstyle{remark}
\newtheorem{remark}{Remark}
\newtheorem{example}{Example}

\newtheorem{code}{Computer Code}

\numberwithin{equation}{section}
\begin{document}
\title[Oeljeklaus-Toma manifolds]{Oeljeklaus-Toma manifolds and arithmetic invariants}
\author{O. Braunling}
\address{Albert Ludwig University of Freiburg, Eckerstra\ss e 1, D-79104 Freiburg, Germany}
\thanks{The author has benefitted from the GK1821 \textquotedblleft Cohomological
Methods in Geometry\textquotedblright.}
\date{{\today }}

\begin{abstract}
We consider Oeljeklaus-Toma manifolds coming from number fields with precisely
one complex place. Our general theme is to relate the geometry with the
arithmetic. We show that just knowing the fundamental group allows to recover
the number field. We also show that this fails if there are more complex
places. The first homology turns out to relate to an interesting ideal. We
compute the volume in terms of the discriminant and regulator of the number
field. Is there a conceptual reason for this? We explore this and see what
happens if we (entirely experimentally!) regard them as \textquotedblleft baby
siblings\textquotedblright\ of hyperbolic manifolds coming from number fields.
We ask the same questions and obtain similar answers, but ultimately it
remains unclear whether we are chasing ghosts or not.

\end{abstract}
\maketitle

Unless stated otherwise, in this note the term \textquotedblleft
Oeljeklaus-Toma manifold\textquotedblright\ refers to the construction
$X:=X(K;\mathcal{O}_{K}^{\times,+})$ \cite[\S 1]{MR2141693} for an arbitrary
number field $K$ with $s\geq1$ real places and \textit{precisely one} complex
place:%
\[
X:=\left(  \mathbf{H}^{s}\times\mathbf{C}\right)  /(\mathcal{O}_{K}%
\rtimes\mathcal{O}_{K}^{\times,+})\text{,}%
\]
where $\mathbf{H}$ refers to the complex upper half plane. Such a space $X$
cannot carry a K\"{a}hler metric, but becomes a locally conformally K\"{a}hler
(= LCK) manifold by a metric found in \cite{MR2141693}. These generalize the
Inoue surfaces of type $\mathrm{S}^{0}$ \cite[\S 2]{MR0342734}. LCK manifolds
are an exciting generalization of K\"{a}hler manifolds and the Oeljeklaus-Toma
manifolds supply a load of non-K\"{a}hler examples with rich additional
properties \cite{MR3195237}. But even more excitingly, there is already ample
evidence that their complex geometry truly reflects arithmetical properties of
the number field $K$, e.g. \cite{MR2875828}, \cite{MR3236651},
\cite{MR3193953}. This note intends to show that this also applies to the
volume and fundamental group.

For any number field $K$ let $K^{\operatorname*{g}}$ be its Galois closure
over the rationals; for a group $\pi$ write $\pi_{\operatorname*{ab}}%
:=\pi/[\pi,\pi]$ for its maximal abelian quotient, $\pi_{\operatorname*{tor}}$
for its torsion subgroup, and $\pi_{\operatorname*{fr}}:=\pi/\pi
_{\operatorname*{tor}}$ for its maximal torsion-free quotient.\medskip

Firstly, we can reconstruct the number field $K$ just from
the fundamental group of $X$ by the following recipe:

\begin{proposition}
\label{prop_reconstruct}Let $X$ be an Oeljeklaus-Toma manifold, but it
suffices to know $\pi:=\pi_{1}(X,\ast)$. Then we have the short exact sequence%
\begin{equation}
1\longrightarrow\varkappa\longrightarrow\pi\longrightarrow\pi
_{\operatorname*{ab},\operatorname*{fr}}\longrightarrow1\text{,}\label{ly1}%
\end{equation}
where $\varkappa$ is just defined as the kernel. There is a (non-canonical)
isomorphism $\varkappa\simeq\mathbf{Z}^{n}$ for some $n$. Sequence \ref{ly1}
induces an action $\pi_{\operatorname*{ab},\operatorname*{fr}}%
\circlearrowright\varkappa$ and therefore a representation%
\[
\rho:\pi_{\operatorname*{ab},\operatorname*{fr}}\longrightarrow
\operatorname*{GL}\nolimits_{n}(\mathbf{Z})\text{,}%
\]
well-defined up to conjugation. Then $K$ is uniquely determined and
$K^{\operatorname*{g}}=\mathbf{Q}\left(  \{\lambda\}\right)  $, where
$\{\lambda\}$ is the set of complex eigenvalues of $\rho$.
\end{proposition}

Quite differently, if we take the spaces $X(K;U)$ of \cite{MR2141693} with
$t>1$ complex places into consideration, we will exhibit two diffeomorphic
such manifolds whose underlying number fields are different, even after taking
the Galois closure. So no reconstruction of $K^{\operatorname*{g}}$ as above
can be possible for $t>1$. See Example
\ref{Example_CannotReconstructForMoreComplexPlaces} in the text.\medskip

Secondly, it was already shown in \cite{MR2141693} that the first Betti number
of $X$ agrees with the number of real places of $K$. It turns out that the
torsion in the first homology group carries also some rather subtle arithmetic information:

\begin{proposition}
\label{prop_TorsionInH1}Let $X$ be an Oeljeklaus-Toma manifold. Then there is
a canonical isomorphism%
\[
H_{1}(X,\mathbf{Z})_{\operatorname*{tor}}\overset{\sim}{\longrightarrow
}\mathcal{O}_{K}/J(\mathcal{O}_{K}^{\times,+})\text{,}%
\]
where $J(\mathcal{O}_{K}^{\times,+})$ is the ideal generated by all $1-u$ with
$u\in\mathcal{O}_{K}^{\times,+}$. In particular, $H_{1}(X,\mathbf{Z}%
)_{\operatorname*{tor}}$ is generated by at most $s+2$ elements.
\end{proposition}

See Prop. \ref{Prop_KappaAgreesWithOK} in the text for a more general version.
Although not spelled out explicitly in their article \cite{MR2875828}, this result
follows directly from the ideas of Parton and Vuletescu. We give a
number of examples where this torsion subgroup is quite large, notably%
\[
\#H_{1}(X,\mathbf{Z})_{\operatorname*{tor}}=2^{2}\cdot5^{2}\cdot7\cdot
967\cdot1649120827309715616889\text{.}%
\]
See Example \ref{example_ComputeJ2} below. We also show that for any given
$m\geq1$ there exists an Inoue surface of type \textrm{S}$^{0}$ with
$H_{1}(X,\mathbf{Z})_{\operatorname*{tor}}\cong\mathbf{Z}/m$, see Prop.
\ref{prop_ConstructInoueSurfaceWithTorsionZm}. As far as I can tell, there is
no easy general formula to compute the order of the torsion group. However, a
computer can determine it rather quickly in any given case.\medskip

Thirdly, if we use the canonical metric on $X$, the volume of $X$ relates $-$
in a rather simple way $-$ to the picture of Dirichlet's analytic class number
formula. This leads to a connection to the zeta function of the number field, namely:

\begin{proposition}
\label{prop_main}Let $K$ be a number field with $s\geq1$ real places and one
complex place and $X:=X(K;\mathcal{O}_{K}^{\times,+})$ the associated
Oeljeklaus-Toma manifold equipped with its canonical metric. Then%
\begin{align}
\operatorname*{Vol}\left(  X\right)   &  =\frac{(s+1)}{4^{s}\cdot2^{s^{2}}%
}\cdot\sqrt{\left\vert \triangle_{K/\mathbf{Q}}\right\vert }\cdot
R_{K}\label{l6}\\
&  =\frac{(s+1)}{2^{3s+s^{2}}h_{K}}\cdot\left\vert \triangle_{K/\mathbf{Q}%
}\right\vert \cdot\pi^{-1}\cdot\operatorname*{res}\nolimits_{s=1}\zeta
_{K}(s)\text{,} \label{lcy3}%
\end{align}
where $\triangle_{K/\mathbf{Q}}$ denotes the discriminant of $K$, $R_{K}$ is
the Dirichlet regulator, $h_{K}$ the class number of $K$, $\zeta_{K}$ the
Dedekind zeta function of $K$.
\end{proposition}

The essence of this computation is Equation \ref{l6}. The second line is just
a reformulation in terms of the analytic class number formula. The volumes
behave in a certain fashion, quite similar to analogous results in the theory
of hyperbolic manifolds:

\begin{proposition}
\label{prop_boundedminvolume}Among all Oeljeklaus-Toma manifolds $X$ of some
fixed dimension, there is a smallest possible volume, realized by at least
one, but at most finitely many Oeljeklaus-Toma manifolds. There is a unique
smallest Oeljeklaus-Toma surface, its volume is%
\[
\operatorname*{Vol}\left(  X\right)  =0.337146\ldots
\]
It is the one coming from the number field%
\[
K:=\mathbf{Q}[T]/(T^{3}-T+1)\text{.}%
\]

\end{proposition}

With the assistance of the computer we can also provide the smallest possible
Oeljeklaus-Toma manifolds in dimensions $\leq6$. In all these cases the
smallest manifold always comes from the number field with given $s$ and $t=1$
whose discriminant has the smallest absolute value. The reader should not
hasten to believe that this might be a general fact. There are pairs with
increasing discriminant, while the volume decreases. See
\S \ref{sect_SmallestVolume}.\medskip

We should also explain why one could expect the volume of an Oeljeklaus-Toma
manifold to be related to the arithmetic of $K$ at all $-$ by parallels to a
much more complicated, but also much older theory:

\section{\label{sect_toymodelproducthyperbolic}Parallels to arithmetic
hyperbolic manifolds}

\subsection{Phenomenology}

We want to look at the geometry of Oeljeklaus-Toma manifolds by investigating
to what extent there exist some parallels to the behaviour of hyperbolic
manifolds in dimension $\geq3$. In the following list it might not be clear
(especially to the author of these lines) which analogies are meaningful, and
which are accidents\footnote{Especially since mathematics does not have
accidents. It has, however, plenty of room for misleading analogies.}:

\begin{enumerate}
\item Going through the universal covering space, we have the presentation%
\[
X=\mathbf{H}_{n}/\Gamma\qquad\text{versus}\qquad X^{\prime}=(\mathbf{H}%
^{s}\times\mathbf{C})/\Gamma^{\prime}\text{,}%
\]
where $X$ denotes a hyperbolic $n$-manifold; $\mathbf{H}_{n}$ denotes
hyperbolic $n$-space with the hyperbolic metric. On the right hand side
$X^{\prime}$ is an Oeljeklaus-Toma manifold and $\mathbf{H}^{s}\times
\mathbf{C}$ is equipped with the Oeljeklaus-Toma locally conformally
K\"{a}hler metric. In either case, $\Gamma$ is a discrete, finite covolume
subgroup of the relevant isometry groups.

\item In dimension $\geq3$ finite volume hyperbolic manifolds are determined
(up to isometry) by their fundamental groups (Mostow-Prasad Rigidity). For
Oeljeklaus-Toma manifolds we find in Proposition \ref{prop_reconstruct} that
they are also uniquely determined by their fundamental groups. Up to
diffeomorphism this is just Mostow Rigidity for real solvmanifolds, but we
even get the number field $K$ back.

\item As a consequence of Rigidity, the volume of a hyperbolic $n$-manifold
for $n\geq3$ is a \textit{topological} invariant. Similarly, once we pick an
overall normalization for the K\"{a}hler potential on $\mathbf{H}^{s}%
\times\mathbf{C}$, diffeomorphic Oeljeklaus-Toma manifolds admit a canonical
notion of volume.

\item Among the hyperbolic manifolds, there is the special class of
`arithmetic' ones. They come from arithmetic groups defined through number
fields $K$, and their volume relates to the special value $\zeta_{K}(2)$ of
the zeta function of the number field. For example, there is Humbert's
formula, discovered in 1919,%
\begin{equation}
\operatorname*{Vol}\left(  \mathbf{H}_{3}/\operatorname*{PSL}\nolimits_{2}%
\left(  \mathcal{O}_{K}\right)  \right)  =\frac{1}{4}\cdot\left\vert
\triangle_{K/\mathbf{Q}}\right\vert ^{\frac{3}{2}}\cdot\pi^{-2}\cdot\zeta
_{K}(2) \label{ly8}%
\end{equation}
for $K$ an imaginary quadratic number field. This can be generalized broadly,
for example encompassing number fields with $s\geq1$ real and $t=1$ complex
places. One switches to product-hyperbolic geometries,%
\[
\Gamma\circlearrowright\mathbf{H}^{s}\times\mathbf{H}_{3}^{t}\text{,}%
\qquad\qquad X:=\left(  \mathbf{H}^{s}\times\mathbf{H}_{3}^{t}\right)  /\Gamma
\]
and still gets volume formulas of the shape%
\begin{equation}
\operatorname*{Vol}\left(  X\right)  =(\text{rational factor})\cdot\left\vert
\triangle_{K/\mathbf{Q}}\right\vert ^{\frac{3}{2}}\cdot\pi^{-s-2t}\cdot
\zeta_{K}(2)\text{.} \label{ly3}%
\end{equation}
See \cite[Thm. 3.2]{MR3117524} for a whole panorama of related volume
computations. On the other hand, Proposition \ref{prop_main} shows that the
volume of Oeljeklaus-Toma manifolds is%
\[
(\text{rational factor})\cdot\pi^{-1}\cdot\operatorname*{res}\nolimits_{s=1}%
\zeta_{K}(s)\text{.}%
\]
This is certainly far less exciting than obtaining $\zeta_{K}(2)$, but it is
pleasant that to see that the Oeljeklaus-Toma LCK\ metric, i.e. a metric with
special complex-geometric properties, \textit{entirely naturally} leads to
this kind of volume formula. In some sense, this behaviour was already
built-in within Tricerri's LCK metric for the $\mathrm{S}^{0}$ Inoue surface.

\item Finally, the work of Thurston and Jorgensen \cite{MR636516} has shown a
very interesting structure on the volume distribution among hyperbolic
$3$-manifolds. In particular, there is a unique hyperbolic $3$-manifold of
smallest volume. Asking the same question for Oeljeklaus-Toma manifolds, one
also finds that there is (in each dimension) a smallest volume, attained by at
least one, and at most finitely many Oeljeklaus-Toma manifolds. Incidentally,
the smallest hyperbolic $3$-manifold (the Weeks manifold) is arithmetic, and
comes from the \textit{same} number field as the smallest Oeljeklaus-Toma
manifold. Quite possibly, however, this is just a sporadic effect of small numbers.
\end{enumerate}

\[%
\begin{tabular}
[c]{r|l}%
hyperbolic $3$-manifold & Oeljeklaus-Toma\\\hline
universal cover $\mathbf{H}_{3}$ & universal cover $\mathbf{H}\times
\mathbf{C}$\\
$\operatorname*{SL}\nolimits_{2}(\mathcal{O}_{K})$ & $\supset%
\begin{pmatrix}
a & b\\
& a^{-1}%
\end{pmatrix}
$\\
hyperbolic metric & locally conformally\ K\"{a}hler\\
$\zeta_{K}(2)$ & residue at $\zeta_{K}(1)$,
\end{tabular}
\]
There is one thing I should say very clearly: The above analogies might all be
purely phenomenological. Many of them could be attributed to Oeljeklaus-Toma
manifolds being real solvmanifolds. However, it remains curious that the LCK
metric (i.e. a metric chosen for its special holomorphic features) leads to a
zeta value volume so naturally.

\begin{remark}
On the other hand, it must be said that all of the above analogies completely
collapse if one considers number fields with $t>1$ complex places as well:

\begin{enumerate}
\item In this case the formation of $X(K;U)$ depends on a choice, which is
perhaps un-natural.

\item It seems to be a very delicate question whether there exists some $U$ so
that $X(K;U)$ can be equipped with a LCK metric. See Remark
\ref{rmk_XKU_can_it_be_lck} for recent work on this issue. This makes it
difficult to speak of volume at all.

\item We show in Example \ref{rmk_XKU_can_it_be_lck} that a matching of
fundamental groups, although it still implies being diffeomorphic, does not
allow to reconstruct the field $K$. So even if one can come up with a
normalized LCK metric of some sort, like the one of Battisti \cite[Appendix]%
{MR3193953}, it seems improbable that there is a unique choice within each
diffeomorphism class.
\end{enumerate}
\end{remark}

\section{Preparations}

We shall exclusively use Poincar\'{e}'s upper half plane model for the
hyperbolic $2$-space $\mathbf{H}$. The Iwasawa decomposition of the group
$\operatorname*{SL}\nolimits_{2}(\mathbf{R})$ is the homeomorphism
$\operatorname*{SL}\nolimits_{2}(\mathbf{R})\approx K\times A\times N$, where%
\[
K:=%
\begin{pmatrix}
\cos\theta & -\sin\theta\\
\sin\theta & \cos\theta
\end{pmatrix}
\qquad A:=%
\begin{pmatrix}
a & \\
& a^{-1}%
\end{pmatrix}
\qquad N:=%
\begin{pmatrix}
1 & b\\
& 1
\end{pmatrix}
\text{,}%
\]
(for $a>0$, $b\in\mathbf{R}$) are three subgroups; $K=\operatorname*{SO}%
\nolimits_{2}(\mathbf{R})$ is a maximal compact subgroup. We can use this to
obtain a very pleasant parametrization of the complex upper half plane
$\mathbf{H}:=\{x+iy\mid y>0\}\subset\mathbf{C}$, namely%
\[
A\cdot N=\frac{\operatorname*{SL}\nolimits_{2}(\mathbf{R})}{K}=\frac
{\operatorname*{SL}\nolimits_{2}(\mathbf{R})}{\operatorname*{SO}%
\nolimits_{2}(\mathbf{R})}=\mathbf{H}\text{.}%
\]
Let us recall the details: $\operatorname*{SL}\nolimits_{2}(\mathbf{R})$ acts
on $\mathbf{H}$ via the standard M\"{o}bius action, i.e.%
\[%
\begin{pmatrix}
a & b\\
c & d
\end{pmatrix}
\cdot z:=\frac{az+b}{cz+d}%
\]
for $z\in\mathbf{H}$ a complex number. Since%
\[
A\cdot N=\left\{  \left.
\begin{pmatrix}
a & b\\
& a^{-1}%
\end{pmatrix}
\right\vert a\in\mathbf{R}_{>0}^{\times}\right\}  \text{,}%
\]
the orbit of $i\in\mathbf{H}$ under the action of $A\cdot N$ unwinds as
\begin{equation}
A\cdot N\cong\mathbf{H}\qquad%
\begin{pmatrix}
\sqrt{y} & \frac{x}{\sqrt{y}}\\
& \frac{1}{\sqrt{y}}%
\end{pmatrix}
\cdot i=x+iy\qquad\in\mathbf{H} \label{l2}%
\end{equation}
for $x\in\mathbf{R}$ and $y>0$; this is obviously simply transitive.

\section{Oeljeklaus-Toma manifolds}

Let $K$ be a number field with $s$ real places and $t$ complex places. Let
$n:=[K:\mathbf{Q}]=s+2t$; the ring of integers $\mathcal{O}_{K}$ is a free
$\mathbf{Z}$-module of rank $n$ and the group of units $\mathcal{O}%
_{K}^{\times}$ decomposes as $\mathcal{O}_{K,\operatorname*{tor}}^{\times
}\times\mathbf{Z}^{s+t-1}$, where the torsion subgroup $\mathcal{O}%
_{K,\operatorname*{tor}}^{\times}$ is the group of roots of unity in $K$.
Whenever $s\geq1$, we necessarily have $\mathcal{O}_{K,\operatorname*{tor}%
}^{\times}=\{\pm1\}$. Suppose $\sigma_{1},\ldots,\sigma_{s}:K\rightarrow
\mathbf{R}$ are the real embeddings and $\sigma_{s+1},\ldots,\sigma
_{s+t},\overline{\sigma_{s+1}},\ldots,\overline{\sigma_{s+t}}:K\rightarrow
\mathbf{C}$ the $t$ complex conjugate pairs of complex embeddings (the
numbering and choice of complex conjugate partners is non-canonical, but does
not affect any of the following). Write $\mathcal{O}_{K}^{\times,+}%
:=\{x\in\mathcal{O}_{K}^{\times}\mid\sigma_{j}(x)>0$ for $1\leq j\leq s\}$ for
the group of totally positive units.

Suppose $s\geq1$ and $t\geq1$. Following Oeljeklaus and Toma \cite{MR2141693}
we define%
\begin{equation}
X(K;U):=\frac{\mathbf{H}^{s}\times\mathbf{C}^{t}}{\mathcal{O}_{K}\rtimes
U}\text{,} \label{ly5}%
\end{equation}
where $U\subseteq\mathcal{O}_{K}^{\times,+}$ is a suitably chosen subgroup; it
has to be \textquotedblleft admissible\textquotedblright\ in the sense of
\cite{MR2141693}. The semi-direct product $\mathcal{O}_{K}\rtimes U$ is formed
by letting $U$ act on $\mathcal{O}_{K}$ by multiplication. See \cite[\S 1]%
{MR2141693} for details. It is shown that the action is properly
discontinuous, full rank, and holomorphic. As a result, $X(K;U)$ canonically
becomes a compact complex manifold. In the present text we will mostly deal
with the case $t=1$. It is special in two ways: Firstly, there is a canonical
choice for $U$ because $U:=\mathcal{O}_{K}^{\times,+}$ is always an admissible
subgroup. Secondly, these $X(K;U)$ admit an LCK\ metric. We will review this
in \S \ref{sect_InvariantVolumeForm}.

\begin{remark}
\label{rmk_XKU_can_it_be_lck}In fact by the work of Vuletescu \cite{MR3236651}
and Battisti \cite[Appendix, Theorem 8]{MR3193953}, we now know that $X(K;U)$
as in Equation \ref{ly5} admits an LCK metric if and only if for all
$\alpha\in U$ one has $\left\vert \sigma_{s+1}(\alpha)\right\vert
=\cdots=\left\vert \sigma_{s+t}(\alpha)\right\vert $. For the case we are
mostly interested in, i.e. $t=1$, this condition is trivially met. See
Dubickas \cite{MR3193953} for an extensive study whether this condition can be
satisfied for $t>1$.
\end{remark}

We return to the case of $t=1$ complex places and $U:=\mathcal{O}_{K}%
^{\times,+}$. Firstly, let us point out that the action underlying the
definition of $X(K;\mathcal{O}_{K}^{\times,+})$ can be written in a different fashion:

\begin{lemma}
\label{lemma_sl2_vs_OT_action}On the $\mathbf{H}$-factors in $\mathbf{H}%
^{s}\times\mathbf{C}$ the M\"{o}bius action of the subgroup%
\begin{equation}
\left\{  \left.
\begin{pmatrix}
a & b\\
& a^{-1}%
\end{pmatrix}
\right\vert a\in\mathcal{O}_{K}^{\times,+}\text{ and }b\in\mathcal{O}%
_{K}\right\}  \subseteq\operatorname*{SL}\nolimits_{2}\left(  \mathcal{O}%
_{K}\right)  \label{ly4}%
\end{equation}
under the embedding%
\begin{equation}
\operatorname*{SL}\nolimits_{2}\left(  \mathcal{O}_{K}\right)  \subseteq
\prod_{v\mid\infty}\operatorname*{SL}\nolimits_{2}\left(  K_{v}\right)
\subseteq\underset{\circlearrowright\mathbf{H}^{s}}{\operatorname*{SL}%
\nolimits_{2}\left(  \mathbf{R}\right)  ^{s}}\times\underset{\circlearrowright
\mathbf{C}}{\operatorname*{Aut}(\mathbf{C})} \label{ly6}%
\end{equation}
agrees with the Oeljeklaus-Toma action of the subgroup $\mathcal{O}_{K}%
\rtimes(\mathcal{O}_{K}^{\times,+})^{2}$. In particular, the subgroup in
Equation \ref{ly4} is isomorphic to the semi-direct product $\mathcal{O}%
_{K}\rtimes(\mathcal{O}_{K}^{\times,+})^{2}$.
\end{lemma}

\begin{proof}
This is a straight-forward computation. The isomorphism is induced from the
map%
\begin{align*}
\varphi:\mathcal{O}_{K}\rtimes(\mathcal{O}_{K}^{\times,+})^{2}  &
\longrightarrow\operatorname*{SL}\nolimits_{2}\left(  \mathcal{O}_{K}\right)
\\
(u,v)  &  \longmapsto%
\begin{pmatrix}
\sqrt{v} & \frac{u}{\sqrt{v}}\\
& \frac{1}{\sqrt{v}}%
\end{pmatrix}
\text{.}%
\end{align*}
Note that each unit in $(\mathcal{O}_{K}^{\times,+})^{2}$ has a unique totally
positive square root in $\mathcal{O}_{K}^{\times,+}$ so that the square roots
have a well-defined sense. The multiplication on the left-hand side is
$(u,v)(\tilde{u},\tilde{v})=(u+v\tilde{u},v\tilde{v})$ and we leave it to the
reader to check that $\varphi$ is a group homomorphism. Its image is the group
of Equation \ref{ly4} and an inverse is given by%
\[%
\begin{pmatrix}
a & b\\
& a^{-1}%
\end{pmatrix}
\longmapsto(ab,a^{2})\qquad\in\mathcal{O}_{K}\rtimes(\mathcal{O}_{K}%
^{\times,+})^{2}\text{.}%
\]
Thus, the groups are indeed isomorphic. Let $z:=x+iy\in\mathbf{H}$ be an
arbitrary point, i.e.
\[
x+iy=%
\begin{pmatrix}
\sqrt{y} & \frac{x}{\sqrt{y}}\\
& \frac{1}{\sqrt{y}}%
\end{pmatrix}
\cdot i\text{,}%
\]
following the orbit parametrization coming from Equation \ref{l2}. We compute%
\[%
\begin{pmatrix}
1 & b\\
& 1
\end{pmatrix}
\cdot z=%
\begin{pmatrix}
1 & b\\
& 1
\end{pmatrix}
\cdot%
\begin{pmatrix}
\sqrt{y} & \frac{x}{\sqrt{y}}\\
& \frac{1}{\sqrt{y}}%
\end{pmatrix}
\cdot i=(x+b)+iy
\]
and%
\[%
\begin{pmatrix}
a & \\
& a^{-1}%
\end{pmatrix}
\cdot z=%
\begin{pmatrix}
a & \\
& a^{-1}%
\end{pmatrix}
\cdot%
\begin{pmatrix}
\sqrt{y} & \frac{x}{\sqrt{y}}\\
& \frac{1}{\sqrt{y}}%
\end{pmatrix}
\cdot i=a^{2}(x+iy)
\]
and we immediately see that this agrees under the embedding in Equation
\ref{ly6} with the action of the subgroup $\mathcal{O}_{K}\rtimes
(\mathcal{O}_{K}^{\times,+})^{2}$ as defined in \cite{MR2141693}. Thus, not
only are the groups isomorphic; the Oeljeklaus-Toma action also matches the
M\"{o}bius action.
\end{proof}

\section{\label{section_CommutatorSubgroup}The commutator subgroup $[\pi,\pi
]$}

Next, we want to understand the structure of the commutator subgroup of
$\mathcal{O}_{K}\rtimes U$. Since it will not be any more difficult to treat
the general case, let us for the moment allow for $U$ any subgroup of
$\mathcal{O}_{K}^{\times}$. The following definition is suggested by the
observations in the paper of Parton and Vuletescu \cite[Proof of Thm.
4.2]{MR2875828}:

\begin{definition}
\label{def_IdealJ}For a number field $K$ and an arbitrary subgroup
$U\subseteq\mathcal{O}_{K}^{\times}$ define%
\[
J(U):=\{\text{\emph{ideal generated by }}1-v\text{\emph{\ for all} }v\in
U\}\qquad\subseteq\mathcal{O}_{K}\text{.}%
\]

\end{definition}

Of course this ideal might just be the entire ring of integers. If one
actually wants to compute this ideal, it will be convenient to reduce its
definition to a finite number of generators:

\begin{lemma}
\label{Lemma_JUCanBeDefinedOnGenerators}If $\varepsilon_{1},\ldots
,\varepsilon_{s}$ are generators of the group $U$, then the ideal $J(U)$ can
also be described as%
\[
J(U)=(1-\varepsilon_{1},\ldots,1-\varepsilon_{s})\text{.}%
\]

\end{lemma}

\begin{proof}
It is clear that the $\{1-\varepsilon_{i}\}_{i=1,\ldots,s}$ generate a
sub-ideal of $J(U)$, so it suffices to prove the converse inclusion. Observe
that if $v,\tilde{v}\in U$ then%
\begin{equation}
v(1-\tilde{v})+(1-v)=1-v\tilde{v}\text{,} \label{lcca1}%
\end{equation}
where the left-hand side lies in the ideal generated by $1-v$ and $1-\tilde
{v}$. Therefore, for all products $v=\varepsilon_{1}^{n_{1}}\cdots
\varepsilon_{s}^{n_{s}}$ with $n_{1},\ldots,n_{s}\in\mathbf{Z}$ we may
inductively rewrite $1-v$ along Equation \ref{lcca1} as an element in the
ideal generated by the $\{1-\varepsilon_{i},1-\varepsilon_{i}^{-1}%
\}_{i=1,\ldots,s}$. Here $1-\varepsilon_{i}^{-1}$ occurs since Equation
\ref{lcca1} just allows to reduce products, but not inverses. However, we also
have
\[
1-v^{-1}=-v^{-1}(1-v)
\]
for all units $v\in U$, showing that the generators $1-\varepsilon_{i}^{-1}$
for $i=1,\ldots,s$ are actually not needed.
\end{proof}

\begin{lemma}
For a number field $K$ and subgroups $U,V\subseteq\mathcal{O}_{K}^{\times}$ we
have%
\[
J(U)+J(V)=J(U\cdot V)\text{,}%
\]
where we have a sum of ideals on the left-hand side and $U\cdot V$ denotes the
smallest subgroup of $\mathcal{O}_{K}^{\times}$ containing both $U$ and $V$.
\end{lemma}

\begin{proof}
Every element of $U\cdot V$ has the shape $uv$ with $u\in U,v\in V$ since
$\mathcal{O}_{K}^{\times}$ is abelian. By Equation \ref{lcca1} the element
$1-uv$ lies in the ideal sum $J(U)+J(V)$. Conversely, any element of the sum
can be written as $\sum a_{i}(1-u_{i})$ with $a_{i}\in\mathcal{O}_{K}$ and
$u_{i}$ (for each $i$) either in $U$ or in $V$, so in either case $u_{i}\in
U\cdot V$.
\end{proof}

\begin{proposition}
\label{Prop_StructureOfCommutator}Let $K$ be a number field and $U\subseteq
\mathcal{O}_{K}^{\times}$ a subgroup. Then for the semi-direct product
$\pi:=\mathcal{O}_{K}\rtimes U$ we have%
\[
\lbrack\pi,\pi]=\{(u,1)\in\pi\mid u\in J(U)\}\text{.}%
\]
Analogously, in the subgroup of $\operatorname*{SL}\nolimits_{2}%
(\mathcal{O}_{K})$ defined in Equation \ref{ly4} the commutator subgroup
consists of all matrices%
\[%
\begin{pmatrix}
1 & b\\
& 1
\end{pmatrix}
\]
with $b\in J(U^{2})$.
\end{proposition}

\begin{proof}
We have $(u,v)(\tilde{u},\tilde{v})=(u+v\tilde{u},v\tilde{v})$ and
$(u,v)^{-1}=(-uv^{-1},v^{-1})$. Since the commutator subgroup $[\pi,\pi]$ is
generated by commutators, it is actually generated by all elements of the
shape%
\begin{align}
(u,v)(\tilde{u},\tilde{v})(u,v)^{-1}(\tilde{u},\tilde{v})^{-1}  &
=(u+v\tilde{u},v\tilde{v})(-uv^{-1}-\tilde{u}v^{-1}\tilde{v}^{-1},v^{-1}%
\tilde{v}^{-1})\nonumber\\
&  =((1-\tilde{v})u-(1-v)\tilde{u},1)\text{.} \label{lcca2}%
\end{align}
This shows that the commutator subgroup $[\pi,\pi]$ is contained in the
abelian group $(\mathcal{O}_{K},1)$ and then necessarily a subgroup.
Furthermore, since $u,\tilde{u}\in\mathcal{O}_{K}$ are arbitrary, $[\pi,\pi]$
contains the subgroup $(I,1)$ for $I=(1-\tilde{v},1-v)$, i.e. the ideal in
$\mathcal{O}_{K}$ generated by the elements $1-\tilde{v},1-v$. Since the
latter is true for all $v\in U$ and $[\pi,\pi]$ is closed under addition in
$(\mathcal{O}_{K},1)$, it follows that $[\pi,\pi]$ contains all elements
$(x,1)$ with $x\in J(U)$. Conversely, from Equation \ref{lcca2} it is clear
that all elements in $[\pi,\pi]$ are of the shape $(x,1)$ with $x\in J(U)$,
proving the claim. The claim about $\operatorname*{SL}\nolimits_{2}$ follows
directly from Lemma \ref{lemma_sl2_vs_OT_action}.
\end{proof}

\begin{proposition}
\label{Prop_KappaAgreesWithOK}Let $K$ be a number field and $U\subseteq
\mathcal{O}_{K}^{\times}$ a torsion-free subgroup $\neq1$. Then for
$\pi:=\mathcal{O}_{K}\rtimes U$ the kernel $\varkappa$ in the short exact
sequence%
\begin{equation}
1\longrightarrow\varkappa\longrightarrow\pi\longrightarrow\pi
_{\operatorname*{ab},\operatorname*{fr}}\longrightarrow1 \label{tsy1}%
\end{equation}
is just the subgroup $\mathcal{O}_{K}$. Moreover, if $U$ is admissible in the
sense of \cite{MR2141693}, we have a canonical short exact sequence%
\begin{equation}
0\longrightarrow\mathcal{O}_{K}/J(U)\longrightarrow H_{1}(X(K;U),\mathbf{Z}%
)\longrightarrow U\longrightarrow0\text{,} \label{tsy2}%
\end{equation}
where $\mathcal{O}_{K}/J(U)$ is precisely the torsion subgroup. In particular,
this group needs at most $s+2t$ generators.
\end{proposition}

\begin{proof}
(after Parton and Vuletescu \cite[Thm. 4.2]{MR2875828}) We have the
commutative diagram with exact rows%
\begin{equation}%
%TCIMACRO{\TeXButton{TeX field}{\xymatrix{
%1 \ar[r] & [\pi,\pi] \ar[r] \ar[d] & \pi\ar[r] \ar[d]_{\mathrm{id}}
%& \pi_{\mathrm{ab}} \ar[d] \ar[r] & 1 \\
%1 \ar[r] & \mathcal{O}_{K} \ar[r] & \pi\ar[r] & U \ar[r] & 1,
%}} }%
%BeginExpansion
\xymatrix{
1 \ar[r] & [\pi,\pi] \ar[r] \ar[d] & \pi\ar[r] \ar[d]_{\mathrm{id}}
& \pi_{\mathrm{ab}} \ar[d] \ar[r] & 1 \\
1 \ar[r] & \mathcal{O}_{K} \ar[r] & \pi\ar[r] & U \ar[r] & 1,
}
%EndExpansion
\label{lsn1}%
\end{equation}
where the upper row is formed from abelianization and the lower row from the
semi-direct product structure of $\pi$. The existence of the downward arrows
follows from Prop. \ref{Prop_StructureOfCommutator}, along with $[\pi
,\pi]=J(U)$, where $J(U)$ is an ideal in $\mathcal{O}_{K}$. Since we assume
$U\neq1$, $J(U)$ is not the zero ideal and therefore must have finite index in
$\mathcal{O}_{K}$ as an abelian group; this means that%
\[
m\mathcal{O}_{K}\subseteq\lbrack\pi,\pi]\subseteq\mathcal{O}_{K}%
\]
for some $m$. It is now an easy diagram chase\footnote{The Snake Lemma is
false for arbitrary non-abelian groups, but it \textit{does} hold for the
specific Diagram \ref{lsn1}. The essential reason is that all kernels and
cokernels in this diagram exist. This would not necessarily hold for a general
diagram of non-abelian groups.} to see that $\ker\left(  \pi
_{\operatorname*{ab}}\rightarrow U\right)  $ is a pure torsion group, in fact
it is annihilated by $m$. On the other hand since $U$ is torsion-free, the
kernel of $\pi_{\operatorname*{ab}}\rightarrow U$ must contain the full
torsion subgroup. We conclude that $\ker\left(  \pi_{\operatorname*{ab}%
}\rightarrow U\right)  $ actually agrees with the torsion subgroup in
$\pi_{\operatorname*{ab}}$, and via the snake map with $\mathcal{O}_{K}/J(U)$.
Since the right-hand side downward arrow in Diagram \ref{lsn1} is moreover
surjective, we deduce that $U=\pi_{\operatorname*{ab},\operatorname*{fr}}$ and
therefore $\varkappa$ in Equation \ref{tsy1} agrees with $\mathcal{O}_{K}$.
Finally, if $U$ is admissible, we can form $X(K;U)$ and by the Hurewicz
Theorem there is a canonical isomorphism%
\[
\pi_{1}(X(K;U),\ast)_{\operatorname*{ab}}\overset{\sim}{\longrightarrow}%
H_{1}(X(K;U),\mathbf{Z})
\]
and our previous argument decomposes the left-hand side just in the shape of
Equation \ref{tsy2}. Since $\mathcal{O}_{K}\simeq\mathbf{Z}^{s+2t}$, every
quotient group requires at most $s+2t$ generators itself.
\end{proof}

We should give some examples regarding the structure of $J(U)$. Firstly, it is
easy to produce examples where the ideal is non-trivial:

\begin{example}
\label{example_ComputeJ2}We give a few examples for the group of totally
positive units, i.e. we consider the ideal norms $N(J(\mathcal{O}_{K}%
^{\times,+}))=\#\mathcal{O}_{K}/J(\mathcal{O}_{K}^{\times,+})$. We perform
this computation for the number fields%
\[
F_{m}=\mathbf{Q}[T]/(T^{3}-T+m)\quad G_{m}=\mathbf{Q}[T]/(T^{7}-T-m)\quad
H_{m}=\mathbf{Q}[T]/(T^{3}-2T-m)\text{,}%
\]
the result spelled out in the respective column:%
\[%
\begin{tabular}
[c]{c|c|cc}%
$m$ & $F_{m}$ & $G_{m}$ & \multicolumn{1}{|c}{$H_{m}$}\\\hline
\multicolumn{1}{r|}{$1$} & \multicolumn{1}{|r|}{$1$} &
\multicolumn{1}{|r|}{$1$} & \multicolumn{1}{r}{$-$}\\
\multicolumn{1}{r|}{$2$} & \multicolumn{1}{|r|}{$2^{2}$} &
\multicolumn{1}{|r|}{$2^{2}$} & \multicolumn{1}{r}{$2$}\\
\multicolumn{1}{r|}{$3$} & \multicolumn{1}{|r|}{$3^{2}$} &
\multicolumn{1}{|r|}{$1$} & \multicolumn{1}{r}{$2\cdot3^{2}$}\\
\multicolumn{1}{r|}{$4$} & \multicolumn{1}{|r|}{$2^{3}$} &
\multicolumn{1}{|r|}{$2^{2}$} & \multicolumn{1}{r}{$-$}\\
\multicolumn{1}{r|}{$5$} & \multicolumn{1}{|r|}{$19$} &
\multicolumn{1}{|r|}{$1$} & \multicolumn{1}{r}{$2^{4}$}\\
\multicolumn{1}{r|}{$6$} & \multicolumn{1}{|r|}{$-$} &
\multicolumn{1}{|r|}{$2^{2}$} & \multicolumn{1}{r}{$2\cdot3\cdot5\cdot11$}\\
\multicolumn{1}{r|}{$7$} & \multicolumn{1}{|r|}{$17$} &
\multicolumn{1}{|r|}{$1$} & \multicolumn{1}{r}{$2\cdot7\cdot109$}%
\end{tabular}
\]
This table was generated by computer, see Code \ref{comp_code_ComputeIdealJ}
on page \pageref{comp_code_ComputeIdealJ} for details. The dashes
\textquotedblleft$-$\textquotedblright\ indicate whenever the given polynomial
is not irreducible. Solely for the entertainment of the reader, let us also
list a large example: For the randomly chosen number field $\mathbf{Q}%
[T]/(T^{3}+2T+2000)$ one gets%
\[
\#\mathcal{O}_{K}/J(\mathcal{O}_{K}^{\times,+})=2^{2}\cdot5^{2}\cdot
7\cdot967\cdot1649120827309715616889\text{.}%
\]
As far as I can tell, there does not seem to be an obvious pattern governing
the structure of the ideal $J(U)$.
\end{example}

\begin{remark}
\label{remark_MakeJBeOne}If the reader wants to produce infinite families of
number fields so that $J(\mathcal{O}_{K}^{\times,+})=(1)$ for all of them, the
easiest way is to pick some number field $L$ with $J(\mathcal{O}_{L}%
^{\times,+})=(1)$. Then take any family $L_{i}$ of number fields and consider
the composita $K_{i}:=L_{i}\cdot L$. Then we have%
\[
\mathcal{O}_{K_{i}}=J_{L_{i}}(\mathcal{O}_{L_{i}}^{\times,+})\mathcal{O}%
_{K_{i}}\subseteq J_{K_{i}}(\mathcal{O}_{L_{i}}^{\times,+})\subseteq J_{K_{i}%
}(\mathcal{O}_{K_{i}}^{\times,+})\text{,}%
\]
where $J_{F}$ refers to forming the ideal $J(U)$ with respect to the number
field $F$.
\end{remark}

Of course Proposition \ref{Prop_KappaAgreesWithOK} provokes the question: What
abelian groups can occur for $\mathcal{O}_{K}/J(U)$ at all? For example, for
$s=t=1$ we see that only finite abelian groups with at most three generators
are possible; this is of course already in Inoue's original paper \cite[\S 2,
p. 274]{MR0342734}. Do all of them really occur? I supply a crude `first
approach' for cyclic groups in Prop.
\ref{prop_ConstructInoueSurfaceWithTorsionZm} below, but it is not quite
satisfactory.\medskip

There is also a completely different way to characterize $\mathcal{O}_{K}$
inside $\mathcal{O}_{K}\rtimes U$ and it could serve as an alternative
definition of $\varkappa$ in the formulation of Proposition
\ref{prop_reconstruct}:

\begin{proposition}
\label{prop_IdentifyOKAsLargestZnSubgroup}Let $K$ be a number field and
$U\subseteq\mathcal{O}_{K}^{\times}$ a subgroup. Consider the family of all
subgroups%
\[
\mathcal{H}:=\left\{  H\subseteq\mathcal{O}_{K}\rtimes U\mid H\simeq
\mathbf{Z}^{n}\text{ for some }n\geq0\right\}  \text{.}%
\]
Then there is a maximal $n$ which can occur, and all those realizing the
maximal $n$ are partially ordered by inclusion and there is a unique maximal
$H$ among them. In fact, this maximal $H$ is the subgroup $\mathcal{O}_{K}$.
\end{proposition}

\begin{proof}
Suppose some non-trivial $H\in\mathcal{H}$ exists and let $(u,v)\in H$ be some
element which is not the identity. Since $H$ is abelian, all $(\tilde
{u},\tilde{v})\in H$ must commute with $(u,v)$. By Equation \ref{lcca2} this
forces%
\[
\mathbf{1}_{H}=(u,v)(\tilde{u},\tilde{v})(u,v)^{-1}(\tilde{u},\tilde{v}%
)^{-1}=((1-\tilde{v})u-(1-v)\tilde{u},1)\text{,}%
\]
so%
\begin{equation}
(1-\tilde{v})u-(1-v)\tilde{u}=0\text{.} \label{lT2}%
\end{equation}
Now we need a case distinction:\newline(1) Suppose $v\neq1$. Then in the field
$K$ we can solve for $\tilde{u}$ and find%
\begin{equation}
\tilde{u}=\frac{1-\tilde{v}}{1-v}u\text{.} \label{lT1}%
\end{equation}
For any given $\tilde{v}\in U$ it can be true or false that the right-hand
side lies in $\mathcal{O}_{K}$ (recall that $u,v$ are fixed). We obtain that
the largest subset of $\mathcal{O}_{K}\rtimes U$ of elements commuting with
$(u,v)$ is%
\begin{equation}
C_{u,v}:=\left\{  (\tilde{u},\tilde{v})\text{ so that }\tilde{u}%
=\frac{1-\tilde{v}}{1-v}u\in\mathcal{O}_{K}\right\}  \label{lcca3}%
\end{equation}
and by definition as a centralizer this is actually a subgroup. The latter can
also be checked directly. As $H$ is abelian and contains $(u,v)$, we must have
$H\subseteq C_{u,v}$. We compose the inclusion of $H$ with the projection of
the semi-direct product, i.e.%
\begin{align*}
H\hookrightarrow C_{u,v}\hookrightarrow\mathcal{O}_{K}\rtimes U  &
\longrightarrow U\hookrightarrow\mathcal{O}_{K}^{\times}\simeq\mu_{K}%
\times\mathbf{Z}^{s+t-1}\\
(\tilde{u},\tilde{v})  &  \longmapsto\tilde{v}\text{.}%
\end{align*}
Since Equation \ref{lT1} implies that $\tilde{u}$ can be computed from
$\tilde{v}$ (and $u,v$ were fixed), this composition is actually injective. It
follows that the $\mathbf{Z}$-rank of $H$ can be at most $s+t-1$.\newline(2)
Suppose $v=1$. Then Equation \ref{lT2} becomes $(1-\tilde{v})u=0$. We have
$u\neq0$ since $(u,v)=(0,1)$ would then be the identity element, which we had
excluded. Hence, the elements in $\mathcal{O}_{K}\rtimes U$ commuting with
$(u,v)$ are precisely those with $\tilde{v}=1$; and these are precisely those
forming the subgroup $\mathcal{O}_{K}\simeq\mathbf{Z}^{s+2t}$. We always have
$s+2t>s+t-1$, so we conclude that the subgroups of Equation \ref{lcca3} are
never realizing the maximal $\mathbf{Z}$-rank. Instead, it follows that only
those $H\subseteq\mathcal{O}_{K}$ with $H\simeq\mathbf{Z}^{s+2t}$ realize the
maximal $\mathbf{Z}$-rank and among all such $H$ contained in $\mathcal{O}%
_{K}$ clearly the full group $\mathcal{O}_{K}\simeq\mathbf{Z}^{s+2t}$ is the
unique maximal.
\end{proof}

\section{Proof of Proposition \ref{prop_reconstruct}}

This is inspired from the perspective taken in the proof of \cite[Thm.
4.2]{MR2875828} by Parton and Vuletescu.

\begin{proof}
[Proof of Prop. \ref{prop_reconstruct}]Since $X$ is an Oeljeklaus-Toma
manifold, it is connected and therefore $\pi:=\pi_{1}(X,x)$ is well-defined up
to the inner automorphism coming from the choice of picking a base point. We
denote by $\pi_{\operatorname*{ab}}$ its maximal abelian quotient, and by
$\pi_{\operatorname*{ab},\operatorname*{fr}}$ the maximal torsion-free
quotient of the latter. We may then define $\varkappa$ just as the
corresponding kernel in the short exact sequence%
\begin{equation}
1\longrightarrow\varkappa\longrightarrow\pi\longrightarrow\pi
_{\operatorname*{ab},\operatorname*{fr}}\longrightarrow1\text{.}\label{ltt7}%
\end{equation}
Since $\mathbf{H}^{s}\times\mathbf{C}$ is contractible and the action of
$\mathcal{O}_{K}\rtimes\mathcal{O}_{K}^{\times,+}$ is easily checked to be
free, the Oeljeklaus-Toma manifold $X=(\mathbf{H}^{s}\times\mathbf{C}%
)/(\mathcal{O}_{K}\rtimes\mathcal{O}_{K}^{\times,+})$ is actually a
classifying space for the group, i.e. its homotopy type is a $B(\mathcal{O}%
_{K}\rtimes\mathcal{O}_{K}^{\times,+})=K(\mathcal{O}_{K}\rtimes\mathcal{O}%
_{K}^{\times,+},1)$, an Eilenberg-MacLane space: $\pi_{1}(X,x)\simeq
\mathcal{O}_{K}\rtimes\mathcal{O}_{K}^{\times,+}$ (and moreover $\pi
_{i}(X,x)=0$ for $i\geq2$, but we do not need this). By Prop.
\ref{Prop_KappaAgreesWithOK} we already know that for the group $\pi
=\mathcal{O}_{K}\rtimes\mathcal{O}_{K}^{\times,+}$ our Equation \ref{ltt7}
becomes
\[
1\longrightarrow\mathcal{O}_{K}\longrightarrow\pi\longrightarrow
\mathcal{O}_{K}^{\times,+}\longrightarrow1\text{,}%
\]
where $K$ is our still unknown number field. Of course we just know
$\mathcal{O}_{K}$ as an abelian group under addition here; we do not know the
ring structure. However, by the rank of rings of integers and Dirichlet's Unit
Theorem, we know that there exist non-unique isomorphisms%
\begin{equation}
\mathcal{O}_{K}\simeq\mathbf{Z}^{s+2t}\qquad\text{and}\qquad\mathcal{O}%
_{K}^{\times,+}\simeq\mathbf{Z}^{s+t-1}\text{,}\label{ltt8}%
\end{equation}
where $s,t$ are the number of real and complex places of $K$. We have $t=1$ by
assumption\footnote{we had restricted our attention to this case in the entire
text right from the beginning}. We recall that for any short exact sequence of
groups as in Equation \ref{ltt7} there is a morphism $\rho:\pi
_{\operatorname*{ab},\operatorname*{fr}}\rightarrow\operatorname*{Out}%
(\varkappa)$, $\phi\mapsto(x\mapsto\tilde{\phi}x\tilde{\phi}^{-1})$, where
$\tilde{\phi}$ is any lift of $\phi\in\pi_{\operatorname*{ab}%
,\operatorname*{fr}}$ to $\pi$ and \textquotedblleft$\operatorname*{Out}%
$\textquotedblright\ denotes the group of outer automorphisms, i.e.
automorphisms modulo conjugations. Since $\mathcal{O}_{K}$ is abelian and
conjugations are trivial, we may lift this $\rho$ to take values in
$\operatorname*{Aut}(\varkappa)$, and moreover $\rho$ recovers the action used
in the formation of the semi-direct product. From Equation \ref{ltt8} we
already know that $\rho$ becomes%
\[
\rho:\mathcal{O}_{K}^{\times,+}\longrightarrow\operatorname*{GL}%
(\mathbf{Z}^{s+2})
\]
in our particular situation. Choosing a different splitting would just change
$\rho$ into an equivalent representation. Since we know that $\pi
=\mathcal{O}_{K}\rtimes\mathcal{O}_{K}^{\times,+}$ was formed by letting
$\mathcal{O}_{K}^{\times,+}$ act by multiplication on $\mathcal{O}_{K}$, we
know that $\rho$ must (up to conjugation) be precisely the multiplication
action. Hence, the minimal polynomial of any $\alpha\in\mathcal{O}_{K}%
^{\times,+}$ acting on $\mathbf{Z}^{s+2}$ will be nothing but its minimal
polynomial as an algebraic number. A generically chosen $\alpha\in
\mathcal{O}_{K}^{\times,+}$ is a primitive element for the field extension
$K/\mathbf{Q}$, so $K$ is uniquely determined. Note moreover that any
conjugation does not affect minimal polynomials, so it does not matter that we
only know $\rho$ up to equivalence. We conclude that all complex eigenvalues
of all $\alpha\in\mathcal{O}_{K}^{\times,+}$ lie in $K^{\operatorname*{g}}$.
Now we are done if adjoining all of them to $\mathbf{Q}$ contains $K$ (because
adjoining all roots of the minimal polynomials produces a Galois extension and
the smallest Galois extension containing $K$ is $K^{\operatorname*{g}}$).
Suppose not. This means that $K$ is a number field such that $\mathcal{O}%
_{K}^{\times,+}$ lies in a proper subfield $L\subsetneqq K$, so even
$\mathcal{O}_{K}^{\times,+}\subseteq\mathcal{O}_{L}^{\times}$, and we get%
\begin{equation}
s^{\prime}+2t^{\prime}<s+2\qquad s\leq s^{\prime}+t^{\prime}-1\label{ltt9}%
\end{equation}
for the real and complex places $s^{\prime},t^{\prime}$ of $L$ by comparing
ranks of rings of integers and units. This forces $t^{\prime}=0$ and
$s^{\prime}=s+1$, so%
\[
s+2=[K:\mathbf{Q}]=[K:L]\cdot\lbrack L:\mathbf{Q}]=[K:L](s+1)\text{,}%
\]
implying $s=0$, which is impossible since our Oeljeklaus-Toma manifolds always
come from number fields with at least one real place.
\end{proof}

\begin{remark}
Instead of identifying $\mathcal{O}_{K}$ inside the fundamental group via the
torsion-free quotient of the abelianization, it can also be characterized as
\textquotedblleft the largest\textquotedblright\ subgroup isomorphic to
$\mathbf{Z}^{n}$ for some $n$. See Proposition
\ref{prop_IdentifyOKAsLargestZnSubgroup} for this alternative perspective.
\end{remark}

There is also the following harmless generalization:

\begin{proposition}
\label{prop_reconstruct_with_U}Let $X$ be known to be of the shape%
\[
X(K;U):=\left(  \mathbf{H}^{s}\times\mathbf{C}\right)  /\left(  \mathcal{O}%
_{K}\rtimes U\right)
\]
with $U$ a finite-index subgroup of $\mathcal{O}_{K}^{\times,+}$ for some
number field $K$ which has $s\geq1$ real places and precisely one complex place.

\begin{enumerate}
\item Then $X(K;U)$ is an LCK manifold.

\item Just knowing its fundamental group $\pi$, $K^{\operatorname*{g}}$ can be
reconstructed from $\pi$ by the same recipe as in Prop. \ref{prop_reconstruct}.

\item We have%
\begin{align*}
&  \sum_{\sigma\in\operatorname*{Gal}(K^{\operatorname*{g}}/\mathbf{Q})}\sigma
U\\
&  \qquad\quad=\left\{  \alpha\in\mathcal{O}_{K^{\operatorname*{g}}}^{\times
}\left\vert
\begin{array}
[c]{l}%
\text{there exists }x\in\pi_{\operatorname*{ab},\operatorname*{fr}}\text{ such
that }\rho(x)\in\operatorname*{GL}\nolimits_{n}(\mathbf{Z})\text{ has}\\
\text{the same minimal polynomial as }\alpha\text{.}%
\end{array}
\right.  \right\}  \text{.}%
\end{align*}
The sum on the left-hand side is the smallest subgroup of $\mathcal{O}%
_{K^{\operatorname*{g}}}^{\times}$ containing $U$ and being closed under the
Galois action of $K^{\operatorname*{g}}/\mathbf{Q}$.
\end{enumerate}
\end{proposition}

\begin{proof}
(1) Once $\mathcal{O}_{K}^{\times,+}$ is admissible, any finite-index subgroup
like $U$ is as well, so $X(K;U)$ is just an instance of the construction in
\cite{MR2141693}. (2) The proof of Prop. \ref{prop_reconstruct} applies word
for word, just replace each $\mathcal{O}_{K}^{\times,+}$ by $U$ and whenever
Dirichlet's Unit Theorem is applied, use that a finite-index subgroup of
$\mathbf{Z}^{n}$ must itself be isomorphic to $\mathbf{Z}^{n}$. (3) For any
$\alpha\in\pi_{\operatorname*{ab},\operatorname*{fr}}\cong U$ the minimal
polynomial of $\rho(\alpha)$ matches the minimal polynomial of $\alpha$ as an
algebraic integer.
\end{proof}

\begin{example}
\label{Example_CannotReconstructForMoreComplexPlaces}In this example we will
show that for $t>1$ complex places, the construction $X(K;U)$ of
\cite{MR2141693} can also produce (non-LCK) complex manifolds with isomorphic
fundamental groups so that the Galois closures of their underlying number
fields differ. Thus, the scope of Proposition \ref{prop_reconstruct} does not
extend to general $X(K;U)$. To this end, consider the number fields%
\[
L_{1}:=\mathbf{Q}[S]/(S^{3}+S+1)\quad L_{2}:=\mathbf{Q}[T]/(T^{3}-T+2)\quad
L_{3}:=\mathbf{Q}[T]/(T^{3}-T+1)\text{.}%
\]
All of these fields satisfy $s=t=1$. Henceforth, we shall also write $S$ and
$T$ to denote the image of these elements in these fields. The element $S$
lies in $\mathcal{O}_{L_{1}}^{\times}$ since its minimal polynomial
$S^{3}+S+1$ has constant coefficient $1$ and therefore must be a unit. We
compute its single real embedding to be $-0.6823\ldots$, so $S^{2}%
\in\mathcal{O}_{L_{1}}^{\times,+}$ generates a subgroup isomorphic to
$\mathbf{Z}$. One can show that $\mathcal{O}_{L_{1}}^{\times,+}=\mathbf{Z}%
\left\langle -S\right\rangle $, but we will not need to know this. Now fix
$i=2,3$ and consider the compositum%
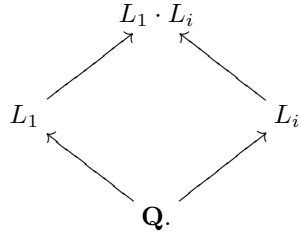
\[%
%TCIMACRO{\TeXButton{TeX field}{\xymatrix{
%& L_{1}\cdot L_{i} \\
%L_{1} \ar[ur] & & L_{i} \ar[ul] \\
%& \mathbf{Q}. \ar[ul] \ar[ur] \\
%}}}%
%BeginExpansion
\xymatrix{
& L_{1}\cdot L_{i} \\
L_{1} \ar[ur] & & L_{i} \ar[ul] \\
& \mathbf{Q}. \ar[ul] \ar[ur] \\
}%
%EndExpansion
\]
We find that $L_{1}\cdot L_{i}$ is a degree $9$ number field with $s=1$ real
places and $t=4$ complex places. Since $T$ is integral, the submodule%
\begin{equation}
J_{i}:=\mathbf{Z}\left\langle 1,S,S^{2},T,T^{2},ST,S^{2}T,ST^{2},S^{2}%
T^{2}\right\rangle \qquad\subset L_{1}\cdot L_{i} \label{lup1}%
\end{equation}
defines a subring of the ring of integers and one checks that we actually have
equality. As far as I can tell confirming this is the only difficult part of
this computation. See Code \ref{comp_code_SageForExampleNoReconstruct} on page
\pageref{comp_code_SageForExampleNoReconstruct} for a verification by
computer. Since our number field has precisely one real place, the admissible
subgroups of $\mathcal{O}_{L_{1}\cdot L_{i}}^{\times,+}$ (in the sense of
\cite{MR2141693}) have rank one and we will simply take $S^{2}\in
\mathcal{O}_{L_{1}}^{\times,+}$ as the generator. The action on the basis
elements of Equation \ref{lup1} is easy to compute, we obtain%
\begin{equation}%
\begin{array}
[c]{lll}%
S\cdot1=S & S\cdot T=ST & S\cdot S^{2}T=(-S-1)T\\
S\cdot S=S^{2} & S\cdot T^{2}=ST^{2} & S\cdot ST^{2}=S^{2}T^{2}\\
S\cdot S^{2}=-S-1 & S\cdot ST=S^{2}T & S\cdot S^{2}T^{2}=(-S-1)T^{2}\text{.}%
\end{array}
\label{lup2}%
\end{equation}
The action of $S^{2}$ follows immediately. The main point here is that this
table does not depend on whether $i=2$ or $3$. For the complex manifolds
$X_{i}:=(\mathbf{H}\times\mathbf{C}^{4})/(\mathcal{O}_{L_{1}\cdot L_{i}%
}\rtimes\left\langle S^{2}\right\rangle )$ of \cite[\S 1]{MR2141693} this
table entirely determines the group structure of the semi-direct product%
\[
\pi_{1}(X_{i},\ast)=J_{i}\rtimes\left\langle S^{2}\right\rangle =\mathbf{Z}%
^{9}\rtimes\mathbf{Z}\text{.}%
\]
Hence, $X_{2}$ and $X_{3}$ have isomorphic fundamental groups. Since they are
actually classifying spaces, it follows that they even have the same homotopy
type. It follows from a result of Oeljeklaus-Toma \cite[Prop. 2.9]{MR2141693}
that the manifolds $X_{i}$ \textsl{do not }admit LCK metrics. They are also
concrete examples which are not `of simple type' (in the sense of
\cite[Definition 1.5 and Remark 1.8]{MR2141693}). Their underlying number
fields have Galois closure $(L_{1}\cdot L_{i})^{\operatorname*{g}}%
=L_{1}^{\operatorname*{g}}\cdot L_{i}^{\operatorname*{g}}$. It is easy to
compute $L_{i}^{\operatorname*{g}}$ for $i=1,2,3$ and we find that each of
them has degree $6$ over $\mathbf{Q}$. We also find that $L_{1}%
^{\operatorname*{g}}\cdot L_{2}^{\operatorname*{g}}\cdot L_{3}%
^{\operatorname*{g}}$ has degree $6^{3}=216$ over $\mathbf{Q}$. This implies
that the degree $36$ fields $L_{1}^{\operatorname*{g}}\cdot L_{2}%
^{\operatorname*{g}}$ and $L_{1}^{\operatorname*{g}}\cdot L_{3}%
^{\operatorname*{g}}$ must be different. See Code
\ref{comp_code_SageForExampleNoReconstruct} on page
\pageref{comp_code_SageForExampleNoReconstruct} for an automated verification
of this example by a computer algebra system.
\end{example}

Going far beyond the case of just Oeljeklaus-Toma manifolds, it is a classical
theorem due to Mostow that any two compact solvmanifolds with isomorphic
fundamental groups must be diffeomorphic \cite[Theorem A]{MR0061611}.

\section{\label{sect_InvariantVolumeForm}The invariant volume form}

Oeljeklaus and\ Toma found a very nice canonical LCK\ metric on
$X(K;\mathcal{O}_{K}^{\times,+})$. See \cite{MR1481969} or \cite{MR0418003}
for an introduction to LCK metrics. This Hermitian metric is induced from a
global K\"{a}hler potential on the universal covering space $\mathbf{H}%
^{s}\times\mathbf{C}=\{(z_{1},\ldots,z_{s+1})\mid z_{1},\ldots,z_{s}%
\in\mathbf{H}$, $z_{s+1}\in\mathbf{C}\}$. The relevant K\"{a}hler potential is%
\begin{equation}
\phi:=\phi_{1}+\left\vert z_{s+1}\right\vert ^{2} \label{lyt1}%
\end{equation}
with%
\[
\phi_{1}:=\prod_{j=1}^{s}\frac{i}{(z_{j}-\overline{z_{j}})}=\frac{1}{2^{s}%
}\prod_{j=1}^{s}\frac{1}{y_{j}}\text{,}%
\]
where $z_{j}=x_{j}+iy_{j}$ for $x_{j},y_{j}\in\mathbf{R}$, cf.
\cite{MR2141693}, modulo the typo corrected in \cite{MR2875828}. Then the
underlying Riemannian metric and the associated $(1,1)$-form are given by%
\begin{equation}
g=\frac{1}{2}%
%TCIMACRO{\tsum _{k,l=1}^{s+1}}%
%BeginExpansion
{\textstyle\sum_{k,l=1}^{s+1}}
%EndExpansion
g_{kl}\left(  \mathrm{d}z_{k}\otimes\mathrm{d}\overline{z_{l}}+\mathrm{d}%
\overline{z_{l}}\otimes\mathrm{d}z_{k}\right)  \qquad\omega=%
%TCIMACRO{\tsum _{k,l=1}^{s+1}}%
%BeginExpansion
{\textstyle\sum_{k,l=1}^{s+1}}
%EndExpansion
\frac{i}{2}g_{kl}\mathrm{d}z_{k}\wedge\mathrm{d}\overline{z_{l}} \label{lyt2}%
\end{equation}
with $g_{kl}:=(\partial_{z_{k}}\partial_{\overline{z_{l}}}\phi)$. One may
follow the efficient explicit computation in \cite[proof of Thm.
5.1]{MR2875828}, leading us to
\[
(g_{kl})=%
\begin{pmatrix}
\frac{\phi_{1}}{2y_{1}^{2}} & \frac{\phi_{1}}{4y_{1}y_{2}} & \frac{\phi_{1}%
}{4y_{1}y_{3}} & \cdots & 0\\
\frac{\phi_{1}}{4y_{1}y_{2}} & \frac{\phi_{1}}{2y_{2}^{2}} &  &  & 0\\
\vdots &  & \ddots &  & \vdots\\
\frac{\phi_{1}}{4y_{1}y_{s}} &  &  & \frac{\phi_{1}}{2y_{s}^{2}} & 0\\
0 & 0 & \cdots & 0 & 2
\end{pmatrix}
\text{.}%
\]
In particular, the determinant of $\left(  g_{kl}\right)  $ is twice the
determinant of the top left $(s\times s)$-minor. For the latter, the Leibniz
formula yields%
\begin{align*}
&  \det(\text{top left }(s\times s)\text{-minor})\\
&  \qquad=\sum_{\sigma\in\Sigma_{s}}\operatorname*{sgn}(\sigma)2^{\#\{j\mid
j=\sigma(j)\}}\frac{\phi_{1}}{4y_{1}y_{\sigma(1)}}\cdots\frac{\phi_{1}}%
{4y_{s}y_{\sigma(s)}}\\
&  \qquad=\frac{\phi_{1}^{s}}{4^{s}}\left(  \sum_{\sigma\in\Sigma_{s}%
}\operatorname*{sgn}(\sigma)2^{\#\{j\mid j=\sigma(j)\}}\right)  \frac{1}%
{y_{1}^{2}\cdots y_{s}^{2}}\text{.}%
\end{align*}
The inner bracket is easily seen to be $s+1$. One way to evaluate this is as
follows: We readily see that the bracket agrees with%
\[
\det%
\begin{pmatrix}
2 & 1 & \cdots & 1\\
1 & 2 & 1 & \vdots\\
\vdots & 1 & 2 & 1\\
1 & \cdots & 1 & 2
\end{pmatrix}
=\left(  -1\right)  ^{s}\det\left(  -\operatorname*{id}-%
\begin{pmatrix}
1 & \cdots & 1\\
\vdots & 1 & \vdots\\
1 & \cdots & 1
\end{pmatrix}
\right)  \text{,}%
\]
so the determinant is nothing but $\left(  -1\right)  ^{s}p_{A}(-1)$ with
$p_{A}(t):=\det(t-A)$ the characteristic polynomial of the matrix $A$ whose
entries are all $1$. The latter has the single non-zero eigenvector
$(1,1,\ldots,1)^{t}$ with eigenvalue $s$. Therefore, the characteristic
polynomial is $(t-s)t^{s-1}$. Hence, $(-1)^{s}p_{A}(-1)=(-1-s)(-1)^{2s-1}%
=(s+1)$, proving the claim. Returning to our main computation,%
\begin{equation}
\det(g_{kl})=2\cdot\frac{(s+1)}{4^{s}}\frac{1}{y_{1}^{2}\cdots y_{s}^{2}}%
\phi_{1}^{s}=\frac{(s+1)}{2^{2s+s^{2}-1}}\frac{1}{y_{1}^{s+2}\cdots
y_{s}^{s+2}}\text{.} \label{ltf1}%
\end{equation}
The K\"{a}hler potential of Equation \ref{lyt1} defines a genuine K\"{a}hler
form on $\mathbf{H}^{s}\times\mathbf{C}$. One easily checks that translations
from $\mathcal{O}_{K}$ leave it invariant, while the multiplication with
elements $\alpha\in\mathcal{O}_{K}^{\times,+}$ changes it by a homothety. More
precisely, from Equation \ref{lyt2} we get%
\[
\omega=\frac{i}{2^{s+3}y_{1}\cdots y_{s}}\left(  \sum_{k,l=1}^{s}%
\frac{2^{\delta_{k=l}}}{y_{k}y_{l}}\mathrm{d}z_{k}\wedge\mathrm{d}%
\overline{z_{l}}\right)  +i\mathrm{d}z_{s+1}\wedge\mathrm{d}\overline{z_{s+1}}%
\]
and therefore%
\begin{align*}
\alpha^{\ast}\omega &  =\frac{i}{2^{s+3}\sigma_{1}(\alpha)\cdots\sigma
_{s}(\alpha)y_{1}\cdots y_{s}}\left(  \sum_{k,l=1}^{s}\frac{2^{\delta_{k=l}}%
}{\sigma_{k}(\alpha)\sigma_{l}(\alpha)y_{k}y_{l}}\sigma_{k}(\alpha
)\overline{\sigma_{l}(\alpha)}\mathrm{d}z_{k}\wedge\mathrm{d}\overline{z_{l}%
}\right) \\
&  +i\left\vert \sigma_{s+1}(\alpha)\right\vert ^{2}\mathrm{d}z_{s+1}%
\wedge\mathrm{d}\overline{z_{s+1}}\text{.}%
\end{align*}
Here $\alpha^{\ast}$ denotes the pullback along the action of $\alpha
\in\mathcal{O}_{K}^{\times,+}$. Note that $\prod_{j=1}^{s+2}\sigma_{j}%
(\alpha)=N(\alpha)=+1$ (usually $\pm1$ since it is a unit, but $+1$ since all
real embeddings $\sigma_{j}(\alpha)>0$ are positive by assumption and the
remaining factor is $\left\vert \sigma_{s+1}(\alpha)\right\vert ^{2}%
=\left\vert \sigma_{s+1}(\alpha)\sigma_{s+2}(\alpha)\right\vert $ of the last
pair of complex embeddings; this is also positive). Hence, $\frac{1}%
{\sigma_{1}(\alpha)\cdots\sigma_{s}(\alpha)}=\left\vert \sigma_{s+1}%
(\alpha)\right\vert ^{2}$, thus $\alpha^{\ast}\omega=\left\vert \sigma
_{s+1}(\alpha)\right\vert ^{2}\omega$. The group $\mathcal{O}_{K}%
\rtimes\mathcal{O}_{K}^{\times,+}$ acts by homotheties on the honest
K\"{a}hler form $\omega$, therefore it does not descend to the quotient and
will not equip it with a K\"{a}hler metric itself, but it means that the
quotient is at least locally conformally K\"{a}hler (LCK). The relevant
invariant form is%
\[
\tilde{\omega}:=y_{1}\cdots y_{s}\omega\qquad\text{i.e.}\qquad\tilde{\omega
}:=e^{f}\omega\text{ with }f:=\log(y_{1}\cdots y_{s})\text{,}%
\]
i.e. this form is invariant under the group action, descends to the
Oeljeklaus-Toma manifold, but is not K\"{a}hler anymore. We compute
$\alpha^{\ast}\tilde{\omega}=\tilde{\omega}$ and $d\tilde{\omega}%
=df\wedge\tilde{\omega}$. In particular, the Lee form \cite[Ch. 1]{MR1481969}
of an Oeljeklaus-Toma manifold is%
\[
\theta:=df=d\log(y_{1}\cdots y_{s})=\sum_{j=1}^{s}d\log(y_{j})=\sum_{j=1}%
^{s}\frac{\mathrm{d}y_{j}}{y_{j}}\text{.}%
\]

\begin{remark}
[Inoue surface of type $\mathrm{S}^{0}$]In the case $s=1$ this simplifies to%
\[
\tilde{\omega}=\frac{i}{8}\frac{\mathrm{d}z_{1}\wedge\mathrm{d}\overline
{z_{1}}}{y_{1}^{2}}+iy_{1}\mathrm{d}z_{2}\wedge\mathrm{d}\overline{z_{2}%
}\text{,}%
\]
which is essentially the $(1,1)$-form associated to the Tricerri metric.
\end{remark}

The Oeljeklaus-Toma manifold comes with a canonical volume form. Just like
$\mathbf{H}^{s}\times\mathbf{C}$ carries the K\"{a}hler form $\omega$ and an
invariant form $\tilde{\omega}$ associated to the LCK metric, $\mathbf{H}%
^{s}\times\mathbf{C}$ has a canonical volume form $vol$ from $\omega$, and an
invariant counterpart $\widetilde{vol}$ belonging to $\tilde{\omega}$. The
volume form can be computed for example by%
\[
vol=\frac{\omega^{s+1}}{(s+1)!}=\left(  \frac{i}{2}\right)  ^{s+1}\det
(g_{kl})\mathrm{d}z_{1}\wedge\mathrm{d}\overline{z_{1}}\wedge\cdots
\wedge\mathrm{d}z_{s+1}\wedge\mathrm{d}\overline{z_{s+1}}\text{,}%
\]
which we can unravel further thanks to Equation \ref{ltf1}. We may now either
switch to $\tilde{\omega}$, or we can equivalently work with the scaled metric
$(y_{1}\cdots y_{s}g_{kl})$ instead of $(g_{kl})$. Then the determinant scales
to $\det(y_{1}\cdots y_{s}\cdot g_{kl})=(y_{1}\cdots y_{s})^{s+1}\det(g_{kl})$
since $(g_{kl})$ is a $(s+1)\times(s+1)$-matrix. Hence, we obtain
\[
\widetilde{vol}:=\left(  \frac{i}{2}\right)  ^{s+1}\frac{(s+1)}{2^{2s+s^{2}%
-1}}\frac{1}{y_{1}\cdots y_{s}}\mathrm{d}z_{1}\wedge\mathrm{d}\overline{z_{1}%
}\wedge\cdots\wedge\mathrm{d}z_{s+1}\wedge\mathrm{d}\overline{z_{s+1}}\text{,}%
\]
which we may rewrite as%
\begin{equation}
\widetilde{vol}=\frac{(s+1)}{2^{2s+s^{2}-1}}\frac{1}{y_{1}\cdots y_{s}%
}\mathrm{d}x_{1}\wedge\mathrm{d}y_{1}\wedge\cdots\wedge\mathrm{d}x_{s+1}%
\wedge\mathrm{d}y_{s+1}\text{.} \label{lcw1}%
\end{equation}
Note that by writing $\frac{\mathrm{d}y}{y}=\mathrm{d}\log y$, this looks just
like the Euclidean volume form in suitable coordinates, say $y:=e^{r}$ or
$e^{2r}$. This is the key reason why Prop. \ref{prop_main} will turn out to be
true.%
\[
=\frac{(s+1)}{2^{2s+s^{2}-1}}\cdot\bigwedge_{j=1}^{s}(\mathrm{d}x_{j}%
\wedge\mathrm{d}\log(y_{j}))\wedge\mathrm{d}x_{s+1}\wedge\mathrm{d}%
y_{s+1}\text{.}%
\]
In the next section we shall tailor a fundamental domain suitable to this
volume form.

\section{A fundamental domain}

Fix a number field $K$ with $s\geq1$ real places and precisely one complex
place.\medskip

In this section we shall determine a fundamental domain for the action of
$\mathcal{O}_{K}\rtimes(\mathcal{O}_{K}^{\times,+})^{2}$ on $\mathbf{H}%
^{s}\times\mathbf{C}$. By Lemma \ref{lemma_sl2_vs_OT_action}\ this action is
precisely the same as the action of the (almost) \textquotedblleft standard
Borel\textquotedblright%
\[%
\begin{pmatrix}
a & b\\
& a^{-1}%
\end{pmatrix}
\subset\operatorname*{SL}\nolimits_{2}(\mathcal{O}_{K})\qquad\text{(with }%
a\in\mathcal{O}_{K}^{\times,+}\text{, }b\in\mathcal{O}_{K}\text{).}%
\]
Here $\operatorname*{SL}\nolimits_{2}(\mathcal{O}_{K})$ acts under the
diagonal embeddings of Equation \ref{ly6}. On the individual factors
$\mathbf{H}$ this is precisely the M\"{o}bius action. We will therefore work
with the coordinates coming from the Iwasawa decomposition of
$\operatorname*{SL}\nolimits_{2}(\mathbf{R})$, see Equation \ref{l2}. It
parametrizes $\mathbf{H}$ in exactly this shape, namely $\mathbf{H}\simeq
A\cdot N$. We use the following explicit coordinates:%
\begin{equation}%
\begin{pmatrix}
e^{r} & b\\
& e^{-r}%
\end{pmatrix}
=%
\begin{pmatrix}
\sqrt{y} & \frac{x}{\sqrt{y}}\\
& \frac{1}{\sqrt{y}}%
\end{pmatrix}
\label{ltz1}%
\end{equation}
with $r,b\in\mathbf{R}$, giving a semi-direct product presentation%
\begin{align}
&  0\longrightarrow\mathbf{R}\longrightarrow A\cdot N\longrightarrow
\mathbf{R}\longrightarrow0\label{ltz2}\\
&  b\mapsto%
\begin{pmatrix}
1 & b\\
& 1
\end{pmatrix}
\text{,}\qquad%
\begin{pmatrix}
e^{r} & b\\
& e^{-r}%
\end{pmatrix}
\mapsto r\text{.}\nonumber
\end{align}
Replicating the decomposition stemming from the coordinates of Equation
\ref{ltz2} for each real place, we get%
\begin{equation}%
\begin{array}
[c]{ccccccc}%
\mathbf{R}^{s}\times\mathbf{C} & \longrightarrow & \mathbf{R} & \times
\cdots\times & \mathbf{R} & \times & \mathbf{C}\\
\downarrow &  & \downarrow &  & \downarrow &  & \downarrow\\
\mathbf{H}^{s}\times\mathbf{C} & \longrightarrow & AN & \times\cdots\times &
AN & \times & \mathbf{C}\\
\downarrow &  & \downarrow &  & \downarrow &  & \\
\mathbf{R}^{s} & \longrightarrow & \mathbf{R} & \times\cdots\times &
\mathbf{R}\text{.} &  &
\end{array}
\label{l3}%
\end{equation}
This picture is also related to the solvmanifold viewpoint proposed by Kasuya
\cite{MR3033950}. Under the diagonal embedding of Equation \ref{ly6}, we have
$\mathcal{O}_{K}\rtimes(\mathcal{O}_{K}^{\times,+})^{2}\subset\prod_{j=1}%
^{s}AN\times\operatorname*{Aut}\mathbf{C}$, and on the factors $AN$ this group
action is just matrix multiplication, see Lemma \ref{lemma_sl2_vs_OT_action}.
Moreover, the arrow%
\begin{equation}%
\begin{array}
[c]{cccc}%
\mathbf{H}^{s}\times\mathbf{C} & \qquad &
\begin{pmatrix}
e^{r_{i}} & b_{i}\\
& e^{-r_{i}}%
\end{pmatrix}
_{i=1,\ldots,s}\times(b_{s+1}) & \\
\downarrow &  & \downarrow & \\
\mathbf{R}^{s} &  & (r_{1},\ldots,r_{s}) &
\end{array}
\label{lq1}%
\end{equation}
is $\mathcal{O}_{K}\rtimes(\mathcal{O}_{K}^{\times,+})^{2}$-equivariant, where
the action on the bottom row unravels to factor through%
\[
\mathcal{O}_{K}\rtimes(\mathcal{O}_{K}^{\times,+})^{2}\twoheadrightarrow
(\mathcal{O}_{K}^{\times,+})^{2}%
\]
and $\alpha\in(\mathcal{O}_{K}^{\times,+})^{2}$ is easily seen to act as
translations%
\[
\alpha\cdot(r_{1},\ldots,r_{s})=(\log\left\vert \sigma_{1}\alpha\right\vert
+r_{1},\ldots,\log\left\vert \sigma_{s}\alpha\right\vert +r_{s})\text{.}%
\]
By Dirichlet's Unit Theorem we can pick free generators $(\mathcal{O}%
_{K}^{\times,+})^{2}=\mathbf{Z}\left\langle \varepsilon_{1},\ldots
,\varepsilon_{s}\right\rangle \simeq\mathbf{Z}^{s}$, i.e. a multiplicatively
independent system of units in this group. Then
\begin{align*}
\Lambda:=  &  \left\{  \sum_{i=1}^{s}\beta_{i}B_{i}\mid0\leq\beta
_{i}<1\right\} \\
B_{i}:=  &  (\log\left\vert \sigma_{1}(\varepsilon_{i})\right\vert
,\ldots,\log\left\vert \sigma_{s}(\varepsilon_{i})\right\vert )^{t}%
\end{align*}
is a fundamental domain for the action of $\mathcal{O}_{K}\rtimes
(\mathcal{O}_{K}^{\times,+})^{2}$ on the base row of Diagram \ref{l3}. Next,
suppose we are given an element in the middle row, say $(b_{1},\ldots
,b_{s},b_{s+1},r_{1},\ldots,r_{s})$ with $b_{1},\ldots,b_{s}\in\mathbf{R}$,
$b_{s+1}\in\mathbf{C}$, $r_{1},\ldots,r_{s}\in\mathbf{R}$. Then for
\textit{fixed} $r_{1},\ldots,r_{s}$ an element $\alpha\in\mathcal{O}%
_{K}\subset\mathcal{O}_{K}\rtimes(\mathcal{O}_{K}^{\times,+})^{2}$ is easily
checked to act as%
\[%
\begin{pmatrix}
1 & \sigma_{i}\alpha\\
& 1
\end{pmatrix}
\cdot%
\begin{pmatrix}
e^{r_{i}} & b_{i}\\
& e^{-r_{i}}%
\end{pmatrix}
=%
\begin{pmatrix}
e^{r_{i}} & b_{i}+e^{-r_{i}}\sigma_{i}(\alpha)\\
& e^{-r_{i}}%
\end{pmatrix}
\]
in the $i$-th coordinate. In particular, the orbit under $\mathcal{O}_{K}$
stays in the same fiber over $r_{1},\ldots,r_{s}$. Fixing the fiber, we see
that $\mathcal{O}_{K}$ acts solely on the coordinates $b_{1},\ldots,b_{s}%
\in\mathbf{R}$ and $b_{s+1}\in\mathbf{C}$ by translation. Moreover, if we pick
generators $\mathcal{O}_{K}=\mathbf{Z}\left\langle a_{1},\ldots,a_{s+2}%
\right\rangle \simeq\mathbf{Z}^{s+2}$ a fundamental domain for the action of
$\mathcal{O}_{K}$ \textit{in the fiber over} $r_{1},\ldots,r_{s}$ is given by%
\begin{align}
\Phi(r_{1},\ldots,r_{s}):=  &  \left\{  \sum_{i=1}^{s+2}\alpha_{i}\tilde
{A}_{i}\mid0\leq\alpha_{i}<1\right\} \label{l5}\\
\text{with }\tilde{A}_{i}:=  &  (e^{-r_{1}}\sigma_{1}(a_{i}),\ldots,e^{-r_{s}%
}\sigma_{s}(a_{i}),\sigma_{s+1}(a_{i}))^{t}\text{.}\nonumber
\end{align}

\begin{proposition}
\label{prop_fund_domain}The set%
\begin{align*}
\mathcal{F}und  &  :=\left\{  \coprod_{(r_{1},\ldots,r_{s})\in\Lambda}^{\cdot
}\Phi(r_{1},\ldots,r_{s})\right\} \\
&  =\left\{  (r_{1},\ldots,r_{s},b_{1},\ldots,b_{s+1})\left\vert
\begin{array}
[c]{l}%
r_{1},\ldots,r_{s}\in\Lambda\\
b_{1},\ldots,b_{s+1}\in\Phi(r_{1},\ldots,r_{s})
\end{array}
\right.  \right\}
\end{align*}
is a fundamental domain for the action of $\mathcal{O}_{K}\rtimes
(\mathcal{O}_{K}^{\times,+})^{2}$ on $\mathbf{H}^{s}\times\mathbf{C}$ in
$(r_{i},b_{i})$-coordinates.
\end{proposition}

\begin{proof}
[Proof of Prop. \ref{prop_fund_domain}]We prove that the inclusion%
\[
\mathcal{F}und\hookrightarrow(\mathbf{R}^{s}\times\mathbf{C})\times
\mathbf{R}^{s}%
\]
induces a bijection onto the quotient by $\mathcal{O}_{K}\rtimes
(\mathcal{O}_{K}^{\times,+})^{2}$.\newline\textit{(Surjectivity)} We have
already observed that the downward arrow in Diagram \ref{lq1} is equivariant.
By Dirichlet's Unit Theorem the $s$ different vectors%
\[
B_{i}^{\prime}:=(\log\left\vert \sigma_{1}(\varepsilon_{i})\right\vert
,\ldots,\log\left\vert \sigma_{s}(\varepsilon_{i})\right\vert ,\log\left\vert
\sigma_{s+1}(\varepsilon_{i})\right\vert )^{t}%
\]
for $i=1,\ldots,s$ give a full rank lattice in the \textquotedblleft
log-norm\textquotedblright\ hyperplane%
\[
H:=\left\{  (v_{1},\ldots,v_{s+1})\mid v_{1}+\cdots+v_{s}+2v_{s+1}=0\right\}
\subseteq\mathbf{R}^{s+1}\text{.}%
\]
Thus, there is a linear isomorphism%
\begin{align*}
\mathbf{R}^{s}  &  \rightarrow H\\
(v_{1},\ldots,v_{s})  &  \mapsto(v_{1},\ldots,v_{s},-\frac{1}{2}(v_{1}%
+\cdots+v_{s}))\\
(v_{1},\ldots,v_{s})  &  \leftarrowtail(v_{1},\ldots,v_{s},v_{s+1})\text{,}%
\end{align*}
where $\mathbf{R}^{s}$ is understood to refer to the base in Diagram
\ref{lq1}. It follows that the $s$ vectors $B_{i}:=(\log\left\vert \sigma
_{1}(\varepsilon_{i})\right\vert ,\ldots,\log\left\vert \sigma_{s}%
(\varepsilon_{i})\right\vert )^{t}$, i.e. just the image of the $B_{i}%
^{\prime}$ under this isomorphism, span a full rank lattice in $\mathbf{R}%
^{s}$. Hence, since the $B_{i}$ are thus an $\mathbf{R}$-vector space basis,
each element in $\mathbf{R}^{s}$ has a unique presentation as%
\[
\sum_{i=1}^{s}(n_{i}+\beta_{i})B_{i}\qquad\text{with}\qquad n_{i}\in
\mathbf{Z}\text{, }0\leq\beta_{i}<1\text{.}%
\]
Thus, letting $\alpha:=\varepsilon_{1}^{-n_{1}}\cdots\varepsilon_{s}^{-n_{s}%
}\in(\mathcal{O}_{K}^{\times,+})^{2}\subset\mathcal{O}_{K}\rtimes
(\mathcal{O}_{K}^{\times,+})^{2}$ act, we obtain an element of the orbit which
lies in our fundamental domain $\Lambda$ for the base. Obviously, since our
map is equivariant, we can let the same uniquely determined element act on the
entire space. Thus, we have found a representative of our element in
$(\mathbf{R}^{s}\times\mathbf{C})\times\Lambda$. Next, the translation action
of $\mathcal{O}_{K}$ leaves the base invariant and just acts in the fibers of
Equation \ref{l3}. We get a unique presentation%
\[
\sum_{i=1}^{s+2}(m_{i}+\alpha_{i})\tilde{A}_{i}\qquad\text{with}\qquad
m_{i}\in\mathbf{Z}\text{, }0\leq\alpha_{i}<1\text{,}%
\]
thus, letting $\beta:=\sum m_{i}b_{i}$ act for $\mathcal{O}_{K}=\mathbf{Z}%
\left\langle a_{1},\ldots,a_{s+2}\right\rangle $, we get a unique
representative in $\Phi(r_{1},\ldots,r_{s})\times\{(r_{1},\ldots,r_{s}%
)\}\in\mathcal{F}und$, as desired. Note that the group elements we acted by
were canonically determined, so we actually get a well-defined map%
\[
(\mathbf{R}^{s}\times\mathbf{C})\times\mathbf{R}^{s}\longrightarrow
\mathcal{F}und\text{.}%
\]
\newline\textit{(Injectivity)} Suppose $x,y\in\mathcal{F}und$ lie in the same
orbit of the action of $\mathcal{O}_{K}\rtimes(\mathcal{O}_{K}^{\times,+}%
)^{2}$. By the equivariance of the morphism in Diagram \ref{lq1} it follows
that their images in $\mathbf{R}^{s}$ lie in the same orbit of the action of
$(\mathcal{O}_{K}^{\times,+})^{2}$ on $\mathbf{R}^{s}$. But since
$x,y\in\mathcal{F}und$, their images $\overline{x},\overline{y}\in
\mathbf{R}^{s}$ lie in $\Lambda$, and since this was a fundamental domain for
$(\mathcal{O}_{K}^{\times,+})^{2}$ we must have $\overline{x}=\overline{y}$.
But then $x,y$ lie in the same fiber $\Phi(r_{1},\ldots,r_{s})$. We check that
$\mathcal{O}_{K}\subseteq\mathcal{O}_{K}\rtimes(\mathcal{O}_{K}^{\times
,+})^{2}$ is the largest subgroup stabilizing a fiber, which implies that
$x,y$ only differ by the translation action of $\mathcal{O}_{K}$ inside the
fiber. But $\Phi(r_{1},\ldots,r_{s})$ was constructed as a fundamental domain
for this action, so we deduce $x=y$.
\end{proof}

We may restate this in more conventional coordinates. Define (or recall) the
standard Minkowski fundamental domain%
\[
\Phi_{\operatorname*{Mink}}:=\left\{  \sum_{i=1}^{s+2}\alpha_{i}\tilde{A}%
_{i}\mid0\leq\alpha_{i}<1\right\}  \subseteq\mathbf{R}^{s}\times\mathbf{C}%
\]
with $\tilde{A}_{i}^{\ast}:=(\sigma_{1}(a_{i}),\ldots,\sigma_{s}(a_{i}%
),\sigma_{s+1}(a_{i}))^{t}$. Just from a change of coordinates Prop.
\ref{prop_fund_domain} can equivalently be reformulated as follows:

\begin{corollary}
\label{cor_fund_domain_in_xy_coords}The set $\Lambda\times\Phi
_{\operatorname*{Mink}}=$%
\[
\left\{  (x_{1},y_{1},\ldots,x_{s},y_{s},x_{s+1}+iy_{s+1})\left\vert
\begin{array}
[c]{l}%
\frac{1}{2}\log y_{1},\ldots,\frac{1}{2}\log y_{s}\in\Lambda\\
x_{1},\ldots,x_{s},x_{s+1}+iy_{s+1}\in\Phi_{\operatorname*{Mink}}%
\end{array}
\right.  \right\}
\]
is the same fundamental domain, but in $(x_{i},y_{i})$-coordinates.
\end{corollary}

\section{The volume computation}

\begin{proof}
[Proof of Prop. \ref{prop_main}]Let us compute the volume of
$X:=X(K;\mathcal{O}_{K}^{\times,+})$. For this we will integrate its canonical
volume form $\widetilde{vol}$ on $X$, which is best done by integrating it
over our fundamental domain of Cor. \ref{cor_fund_domain_in_xy_coords}. We
compute%
\begin{align*}
&  \int_{X}\widetilde{vol}=\frac{1}{2^{s}}\int_{(\mathbf{H}^{s}\times
\mathbf{C})/(\mathcal{O}_{K}\rtimes(\mathcal{O}_{K}^{\times,+})^{2}%
)}\widetilde{vol}=\frac{1}{2^{s}}\int_{\Lambda\times\Phi_{\operatorname*{Mink}%
}}\widetilde{vol}\\
&  =\frac{1}{2^{s}}\int_{\Lambda\times\Phi_{\operatorname*{Mink}}}\frac
{(s+1)}{2^{2s+s^{2}-1}}\frac{1}{y_{1}\cdots y_{s}}\mathrm{d}x_{1}%
\wedge\mathrm{d}y_{1}\wedge\cdots\wedge\mathrm{d}x_{s+1}\wedge\mathrm{d}%
y_{s+1}\text{.}%
\end{align*}
Switching to $r$-coordinates, i.e. substituting $y_{i}=e^{2r_{i}}$ for
$i=1,\ldots,s$, this effectively reduces to computing an Euclidean volume,
namely%
\begin{align*}
&  =\frac{1}{2^{s}}\frac{(s+1)}{2^{2s+s^{2}-1}}\int_{\Lambda\times
\Phi_{\operatorname*{Mink}}}\mathrm{d}x_{1}\wedge\mathrm{d}r_{1}\wedge
\cdots\wedge\mathrm{d}x_{s}\wedge\mathrm{d}r_{s}\wedge\mathrm{d}x_{s+1}%
\wedge\mathrm{d}y_{s+1}\\
\qquad &  =\frac{1}{2^{s}}\frac{(s+1)}{2^{2s+s^{2}-1}}\left(  \int_{\Lambda
}\mathrm{d}x_{1}\wedge\cdots\wedge\mathrm{d}x_{s}\wedge\mathrm{d}x_{s+1}%
\wedge\mathrm{d}y_{s+1}\right)  \left(  \int_{\Phi_{\operatorname*{Mink}}%
}\mathrm{d}r_{1}\wedge\cdots\wedge\mathrm{d}r_{s}\right) \\
&  =\frac{1}{2^{s}}\frac{(s+1)}{2^{2s+s^{2}-1}}\cdot\det\left(  \tilde{A}%
_{1}^{\ast},\ldots,\tilde{A}_{s+2}^{\ast}\right)  \cdot\det(B_{1},\ldots
,B_{s})\text{.}%
\end{align*}

Now we can use the classical fact that the vectors $\tilde{A}_{1}^{\ast
},\ldots,\tilde{A}_{s+2}^{\ast}$, which are generating the Minkowski
fundamental domain, span a parallelepiped of Euclidean volume $2^{-t}%
\cdot\sqrt{\left\vert \triangle_{K/\mathbf{Q}}\right\vert }$ with $t$ the
number of complex embeddings. Moreover, the determinant of the vectors
$B_{1},\ldots,B_{s}$ is almost literally the definition of the Dirichlet
regulator:%
\begin{align*}
&  =\frac{1}{2^{s}}\frac{(s+1)}{2^{2s+s^{2}-1}}\cdot\left(  \frac{1}{2}%
\cdot\sqrt{\left\vert \triangle_{K/\mathbf{Q}}\right\vert }\right)
\cdot(2^{s}\cdot R_{K})\\
&  =\frac{(s+1)}{2^{2s+s^{2}}}\cdot\sqrt{\left\vert \triangle_{K/\mathbf{Q}%
}\right\vert }\cdot R_{K}\text{.}%
\end{align*}
The factor $2^{s}$ in front of the regulator $R_{K}$ occurs as follows: We
would get precisely the regulator here if $\varepsilon_{1},\ldots
,\varepsilon_{s}$ was a basis for $\mathcal{O}_{K}^{\times}$. However, our
$\varepsilon_{1},\ldots,\varepsilon_{s}$ are a basis for $(\mathcal{O}%
_{K}^{\times,+})^{2}$. We recall the analytic class number formula, stating
that (in the case we consider)%
\[
\operatorname*{res}\nolimits_{s=1}\zeta_{K}(s)=\frac{2^{s}\pi h_{K}R_{K}%
}{\sqrt{\left\vert \triangle_{K/\mathbf{Q}}\right\vert }}\text{.}%
\]
Solving for $R_{K}$ yields the claim by plugging it into our previous formula
for the volume. We get Equation \ref{lcy3} as desired.
\end{proof}

One can actually `speed up' this computation slightly by working directly with
a fundamental domain under the action of the full group $\mathcal{O}%
_{K}\rtimes\mathcal{O}_{K}^{\times,+}$, leading the two mutually cancelling
factors $\frac{1}{2^{s}}$ and $2^{s}$ to disappear altogether.

\begin{example}
Consider the cubic field%
\begin{equation}
K:=\mathbf{Q}[T]/(T^{3}+T^{2}-1)\text{.}\label{l35}%
\end{equation}
The image $\overline{T}\in K$ is actually a generator of $\mathcal{O}%
_{K}^{\times,+}$ because its norm is one and its single real embedding has
value $0.7548\ldots>0$. The number field $K$ has discriminant $\triangle
_{K/\mathbf{Q}}=-23$, class number $h_{K}=1$ and regulator%
\[
R_{K}=\left\vert \log0.754877\ldots\right\vert =0.28119957432...
\]
It has one real and one complex place. We may form its Oeljeklaus-Toma
manifold%
\[
X:=X(K;\mathcal{O}_{K}^{\times,+})\text{,}%
\]
giving a (non-K\"{a}hler) Inoue surface of type $\mathrm{S}^{0}$ with
Tricerri's metric. According to Prop. \ref{prop_main} its volume is%
\[
\operatorname*{Vol}\left(  X\right)  =\frac{1}{4}\cdot\sqrt{23}\cdot
0.28119957432...\approx0.3371\ldots\text{.}%
\]
We will show in Prop. \ref{prop_intext_minvol} that no smaller volume is
possible among cubic fields. The commutator subgroup of its fundamental group
is (by Prop. \ref{Prop_StructureOfCommutator})%
\[
\lbrack\pi,\pi]=J(\mathcal{O}_{K}^{\times,+})=(1-\overline{T})\text{,}%
\]
the ideal in $\mathcal{O}_{K}$ generated by $1-\overline{T}$. From the minimal
polynomial, Equation \ref{l35}, we see that $\overline{T}^{3}+\overline{T}%
^{2}=1$ and a simple polynomial division reveals that $(1-\overline{T}%
)\cdot(\overline{T}^{2}+2\overline{T}+2)=1$, showing that $J(\mathcal{O}%
_{K}^{\times,+})=(1)$ is actually the entire ring of integers. So the maximal
abelian quotient $\pi_{\operatorname*{ab}}\simeq\mathbf{Z}$ is already
torsion-free itself. By Prop. \ref{prop_TorsionInH1} we therefore have%
\[
H_{1}(X,\mathbf{Z})=\mathbf{Z}\text{.}%
\]
Following the recipe of Prop. \ref{prop_reconstruct} we let a generator act on
$\mathbf{Z}^{3}$ and this will be $\overline{T}$ or $\overline{T}^{-1}$. We
have no way of distinguishing them if we are just given $\mathbf{Z}$
abstractly. Say it was $\overline{T}$, and we get precisely that its action on
$\mathbf{Z}^{3}$ has minimal polynomial $x^{3}+x^{2}-1$ in $\mathbf{Z}[x]$,
generating $K$ over $\mathbf{Q}$. Adjoining all three complex roots yields
$K^{\operatorname*{g}}/\mathbf{Q}$, a field of degree $6$.
\end{example}

\section{Prescribed torsion}

We want to exhibit a particularly nice family of Inoue surfaces for which we
can freely prescribe the order of the torsion in $H_{1}(X,\mathbf{Z})$. As
will be clear from the proof, this construction largely rests on ideas of
Ishida, porting from number theory to geometry.

\begin{proposition}
\label{prop_ConstructInoueSurfaceWithTorsionZm}For any given $m\geq1$ there
exists an Inoue surface $X$ of type $\mathrm{S}^{0}$ with%
\[
H_{1}(X,\mathbf{Z})\cong\mathbf{Z}\oplus\mathbf{Z}/m
\]
and equipped with the Oeljeklaus-Toma metric, it has volume%
\[
\operatorname*{Vol}(X)=\frac{1}{4}\cdot\sqrt{4m^{3}+27}\cdot\log\left\vert
z-\frac{m}{3z}\right\vert
\]
for the real number%
\[
z:=\sqrt[3]{\frac{1}{2}+\frac{\sqrt{3}}{18}\sqrt{4m^{3}+27}}\text{.}%
\]
In fact, $X$ can be constructed as a finite unramified covering%
\begin{equation}%
%TCIMACRO{\TeXButton{TeX field}{\xymatrix{
%X \ar[d] \\
%X(K;\mathcal{O}_{K}^{\times,+})
%}} }%
%BeginExpansion
\xymatrix{
X \ar[d] \\
X(K;\mathcal{O}_{K}^{\times,+})
}
%EndExpansion
\label{lja21}%
\end{equation}
of the Oeljeklaus-Toma manifold%
\begin{equation}
X:=X(K;\mathcal{O}_{K}^{\times,+})\qquad\text{for}\qquad K:=\mathbf{Q}%
[T]/(T^{3}+mT-1)\text{.} \label{lja20}%
\end{equation}
If $4m^{3}+27$ is square-free, this covering is trivial. Alternatively,
suppose $m=3k$ and $4k^{3}+1$ is square-free: Then if $3\nmid k$, the covering
is also trivial. If $3\mid k$, it is a covering of degree $3$.
\end{proposition}

I suspect that all $H_{1}(X,\mathbf{Z})\cong\mathbf{Z}\oplus\mathbf{Z}/m$ can
be realized by genuine Oeljeklaus-Toma manifolds without the need to allow
finite coverings.

\begin{proof}
Let $m\geq1$ be given. The polynomial $T^{3}+mT-1$ has one sign change in its
coefficients, so by Descartes Sign Rule it has a single positive real root and
no negative real roots. Moreover, it is irreducible over $\mathbf{Q} $
(\textit{Proof:} Otherwise it has a rational root $\alpha_{1}$. Hence, over
the algebraic closure it factors as $(T-\alpha_{1})(T-\alpha_{2})(T-\alpha
_{3})$ with $\alpha_{1}\in\mathbf{Q}\cap\overline{\mathbf{Z}}=\mathbf{Z}$ and
$\alpha_{2},\alpha_{3}\in\overline{\mathbf{Z}}$. Since $\alpha_{1}\alpha
_{2}\alpha_{3}=1$ it follows that $\alpha_{1}$ is also a unit, so $\alpha
_{1}=1$ since we already know that there is no negative real root. But by
plugging in we see that this is certainly not a root). It follows that%
\[
K:=\mathbf{Q}[T]/(T^{3}+mT-1)
\]
is a cubic number field with $s=t=1$. We write $\overline{T}$ to denote the
image of $T$ in $K$. Since the constant coefficient in the minimal polynomial
of $\overline{T}$ is $-1$, it is a unit in $\mathcal{O}_{K}^{\times}$ and we
had already seen that its single real embedding is necessarily positive. Thus,
$\overline{T}\in\mathcal{O}_{K}^{\times,+}$ and it generates a subgroup
$U:=\mathbf{Z}\left\langle \overline{T}\right\rangle $ of finite index.
Similarly, instead of the full ring of integers we so far just understand
$\mathbf{Z}[\overline{T}]\subseteq\mathcal{O}_{K}$, which might be of some
finite index, too.\newline This elementary construction already allows us to
construct $X$: We consider the complex manifold $X$, defined by
\begin{equation}%
\begin{array}
[c]{cl}%
\dfrac{\mathbf{H}\times\mathbf{C}}{\mathbf{Z}[\overline{T}]\rtimes U} & =X\\
\downarrow & \\
\dfrac{\mathbf{H}\times\mathbf{C}}{\mathcal{O}_{K}\rtimes\mathcal{O}%
_{K}^{\times,+}} & =X(K,\mathcal{O}_{K}^{\times,+})
\end{array}
\label{lja10}%
\end{equation}
and equip it with the Oeljeklaus-Toma metric, which is of course also
invariant under the action since $\mathbf{Z}[\overline{T}]\rtimes U$ forms
some finite index subgroup of $\mathcal{O}_{K}\rtimes U$. In particular, $X$
is compact as well. It clearly is an Inoue surface. The definition of the
ideal $J$ (Definition \ref{def_IdealJ}) also makes sense in the subring
$\mathbf{Z}[\overline{T}]\subset\mathcal{O}_{K}$ and we compute%
\begin{align}
J(U)  &  =\mathbf{Z}[\overline{T}]/(\overline{T}-1)\label{lja9}\\
&  =\mathbf{Z}/[T]/(T-1,T^{3}+mT-1)=\mathbf{Z}/m\text{.}\nonumber
\end{align}
We leave it to the reader to check that Prop. \ref{Prop_KappaAgreesWithOK} can
be generalized to the manifold $X$ and gives us $H_{1}(X,\mathbf{Z}%
)\cong\mathbf{Z}\oplus\mathbf{Z}/m$. In fact, the proof carries over verbatim.
Finally, we can compute its volume as follows: Instead of the discriminant
$\triangle_{K/\mathbf{Q}}$ of the number field $K$, we now just get the
discriminant of the order $\mathbf{Z}[\overline{T}]\subset\mathcal{O}_{K}$,
but this makes things easier since that is just the discriminant of the
generating polynomial, i.e. $-4m^{3}-27$. The regulator matrix for $K$ is the
$\left(  1\times1\right)  $-matrix with the single entry $\log\left\vert
\sigma_{1}(\overline{T})\right\vert $, where $\sigma_{1}(\overline{T})$
denotes the single real embedding of $\overline{T}$, or equivalently the
single real root of $T^{3}+mT-1$. We may solve this using the classical Vieta
substitution $t$ for depressed cubics (a variant to the Cardano-Tartaglia
formula): The real root is given by%
\[
t:=z-\frac{m}{3z}\qquad\text{for}\qquad z:=\sqrt[3]{\frac{1}{2}+\frac{\sqrt
{3}}{18}\sqrt{4m^{3}+27}}\text{.}%
\]
This formula is `fairly' simple since in the polynomial $T^{3}+mT-1$ the
quadratic term is already eliminated.\newline The rest of the proof, and the
only difficult part, exclusively concerns the question to control the index of%
\[
\mathbf{Z}[\overline{T}]\rtimes U\subseteq\mathcal{O}_{K}\rtimes
\mathcal{O}_{K}^{\times,+}%
\]
in order to understand the degree of the covering. For the discriminant of the
order $\mathbf{Z}[\overline{T}]\subseteq\mathcal{O}_{K}$ we compute
$\operatorname*{disc}(T^{3}+mT-1)=-4m^{3}-27$ and therefore%
\begin{equation}
-4m^{3}-27=\triangle_{K/\mathbf{Q}}\cdot\lbrack\mathcal{O}_{K}:\mathbf{Z}%
[\overline{T}]]^{2} \label{lja11}%
\end{equation}
by the discriminant-index formula. Hence, if $4m^{3}+27$ is square-free, we
must have $[\mathcal{O}_{K}:\mathbf{Z}[\overline{T}]]=1$ and therefore
$\mathbf{Z}[\overline{T}]=\mathcal{O}_{K}$. Next, we use a clever theorem of
Ishida telling us that this also implies that $\mathcal{O}_{K}^{\times}$ is
generated by $\overline{T}$, namely \cite[Theorem 1]{MR0335469} (strictly
speaking, Ishida's theorem only applies for $m\geq2$, so we ask the reader to
deal with the single case $m=1$ either by using a computer $-$ or by hand. The
latter can be done by checking that the norm equation $N(-)=+1$ cannot have a
real solution of smaller absolute value). The fundamental unit $\overline{T}$
must moreover be totally positive since it was chosen from a polynomial which
did not have negative real roots. Thus,%
\begin{equation}
\mathbf{Z}[\overline{T}]\rtimes U=\mathcal{O}_{K}\rtimes\mathcal{O}%
_{K}^{\times,+} \label{lja17}%
\end{equation}
and the manifold $X$ of Equation \ref{lja10} becomes literally a genuine
Oeljeklaus-Toma manifold.\newline Let us deal with the remaining case: $m=3k$
and $4k^{3}+1$ is square-free. The same theorem of Ishida \cite[Theorem
1]{MR0335469} tells us that this also suffices to have $U=\mathcal{O}%
_{K}^{\times,+}$. However, $[\mathcal{O}_{K}:\mathbf{Z}[\overline{T}]]$ can be
larger than one. Actually, the paper of Ishida gives us also all the tools we
need to deal with this problem, but Ishida does not summarize his findings in
this case as a separate theorem, so let me guide you through his argument: In
\cite[\S 3, all on page $248$]{MR0335469} he first deduces from the
discriminant-index formula, i.e. Equation \ref{lja11}, that%
\[
\lbrack\mathcal{O}_{K}:\mathbf{Z}[\overline{T}]]=3^{d}%
\]
for some $d\geq0$. In the case that $3\nmid k$, he uses that $\overline{T}$
and $\overline{T}+1$ clearly generate the same number field, but%
\[
(T+1)^{3}+m(T+1)-1=T^{3}+3T^{2}+(3k+3)T+3k
\]
is an Eisenstein polynomial at the prime $p=3$, which implies that
$3\nmid\lbrack\mathcal{O}_{K}:\mathbf{Z}[\overline{T}]]$. Thus, again
$\mathbf{Z}[\overline{T}]=\mathcal{O}_{K}$ and we are back in the situation of
Equation \ref{lja17}. It remains to deal with the case $3\mid k$, so
$3^{3}\mid m$. In this case Ishida exhibits the element%
\[
\frac{1}{3}(1+\overline{T}+\overline{T}^{2})\in\frac{1}{3}\mathbf{Z}%
[\overline{T}]\text{,}%
\]
which can be checked by direct computation to be integral, i.e. it lies in
$\mathcal{O}_{K}$. This forces $d\geq1$ and using the discriminant-index
formula once more, he concludes $[\mathcal{O}_{K}:\mathbf{Z}[\overline{T}]]=3
$. Thus, our covering is also of degree $3$.
\end{proof}

\begin{example}
With the help of the computer we can compute the index of $\mathbf{Z}%
[\overline{T}]$ inside $\mathcal{O}_{K}$, resp. $U$ inside $\mathcal{O}%
_{K}^{\times,+}$. Several of the cases below are of course fully explained by
the proposition above. However, not all of them, and in particular we see that
the covering of Equation \ref{lja21} can sometimes have fairly large degree:%
\[%
\begin{tabular}
[c]{c|c|cc}%
$m$ & $[\mathcal{O}_{K}:\mathbf{Z}[\overline{T}]]$ & $[\mathcal{O}_{K}%
^{\times,+}:\mathbf{Z}\left\langle \overline{T}\right\rangle ]$ &
\multicolumn{1}{|c}{$x^{2}\mid4m^{3}+27$}\\\hline
\multicolumn{1}{r|}{$8$} & \multicolumn{1}{|r|}{$5$} &
\multicolumn{1}{|r|}{$2$} & \multicolumn{1}{r}{$5^{2}$}\\
\multicolumn{1}{r|}{$16$} & \multicolumn{1}{|r|}{$1$} &
\multicolumn{1}{|r|}{$1$} & \multicolumn{1}{r}{}\\
\multicolumn{1}{r|}{$24$} & \multicolumn{1}{|r|}{$1$} &
\multicolumn{1}{|r|}{$1$} & \multicolumn{1}{r}{$3^{4}$}\\
\multicolumn{1}{r|}{$32$} & \multicolumn{1}{|r|}{$1$} &
\multicolumn{1}{|r|}{$1$} & \multicolumn{1}{r}{}\\
\multicolumn{1}{r|}{$40$} & \multicolumn{1}{|r|}{$1$} &
\multicolumn{1}{|r|}{$1$} & \multicolumn{1}{r}{}\\
\multicolumn{1}{r|}{$48$} & \multicolumn{1}{|r|}{$1$} &
\multicolumn{1}{|r|}{$1$} & \multicolumn{1}{r}{$3^{2}$}\\
\multicolumn{1}{r|}{$56$} & \multicolumn{1}{|r|}{$31$} &
\multicolumn{1}{|r|}{$2$} & \multicolumn{1}{r}{$31^{2}$}\\
\multicolumn{1}{r|}{$64$} & \multicolumn{1}{|r|}{$1$} &
\multicolumn{1}{|r|}{$1$} & \multicolumn{1}{r}{}\\
\multicolumn{1}{r|}{$72$} & \multicolumn{1}{|r|}{$3\cdot11$} &
\multicolumn{1}{|r|}{$2$} & \multicolumn{1}{r}{$3^{2}\cdot11^{2}$}%
\end{tabular}
\]
The rightmost column lists square factors. Among the first $500$ values of $m
$ we get $\mathcal{O}_{K}=\mathbf{Z}[\overline{T}]$ for $415$ of them. The
condition for $4m^{3}+27$ to be square-free gives a reasonable sufficient
condition to have $[\mathcal{O}_{K}:\mathbf{Z}[\overline{T}]]=1$, but is still
quite remote from a precise criterion. As I have learnt from Ishida's paper
\cite{MR0335469}, it was shown by the famous Erd\H{o}s that $4m^{3}+27$ is
square-free for infinitely many $m$.
\end{example}

\section{A curiosity}

As we had seen from \S \ref{section_CommutatorSubgroup} the structure of the
ideal $J(U)$ can be quite a non-trivial matter. Even though its concrete
structure seems fairly elusive from the outset, one can bound its index in
terms of the units of the underlying number field. Sadly, controlling their
size is similarly inaccessible. However, these two elusive bounds control each
other.\medskip

I only record the following estimate as a curiosity. Since I know of no way to
compute the volume of an Oeljeklaus-Toma manifold except from the arithmetic
invariants of the underlying number field, I would not know how to put the
following inequality into any computational use.

\begin{proposition}
Let $K$ be a number field with $s=t=1$. Then the torsion in the first homology
of the Oeljeklaus-Toma surface $X:=X(K;\mathcal{O}_{K}^{\times,+})$ can be
bounded in terms of the volume and discriminant. Specifically,%
\[
\#H_{1}(X,\mathbf{Z})_{\operatorname*{tor}}\leq3(z+z^{2})
\]
where%
\[
z:=\max(w,\sqrt{1/w})\qquad\text{and}\qquad w:=\exp\left(  4\frac
{\operatorname*{Vol}\left(  X\right)  }{\sqrt{\left\vert \triangle
_{K/\mathbf{Q}}\right\vert }}\right)  \text{.}%
\]

\end{proposition}

\begin{proof}
From Prop. \ref{Prop_KappaAgreesWithOK} we have the equality%
\[
\#H_{1}(X,\mathbf{Z})_{\operatorname*{tor}}=\#(\mathcal{O}_{K}/J(\mathcal{O}%
_{K}^{\times,+}))\text{.}%
\]
By Dirichlet's Unit Theorem $\mathcal{O}_{K}^{\times}\simeq\left\langle
-1\right\rangle \times\mathbf{Z}\left\langle u\right\rangle $ with $u$ a
fundamental unit. Without loss of generality we can assume that $u$ is totally
positive, otherwise replace $u$ by $-u$. Then $u$ is a generator of
$\mathcal{O}_{K}^{\times,+}$. By Lemma \ref{Lemma_JUCanBeDefinedOnGenerators}
we therefore have%
\[
J(\mathcal{O}_{K}^{\times,+})=(1-u)\text{.}%
\]
As this is a principal ideal, its ideal norm can be computed just in terms of
the norm of the generating element.\ This means that%
\[
\#(\mathcal{O}_{K}/J(\mathcal{O}_{K}^{\times,+}))=\left\vert N_{K/\mathbf{Q}%
}(1-u)\right\vert =%
%TCIMACRO{\tprod \nolimits_{i=1}^{3}}%
%BeginExpansion
{\textstyle\prod\nolimits_{i=1}^{3}}
%EndExpansion
\sigma_{i}(1-u)\text{.}%
\]
As usual, let $\sigma_{1}$ denote the single real embedding and $\sigma
_{2},\overline{\sigma_{2}}=\sigma_{3}$ are the complex conjugate embeddings of
the single complex place. We have $\sigma_{1}(u)>0$ and therefore
$N_{K/\mathbf{Q}}(u)=\sigma_{1}(u)\left\vert \sigma_{2}(u)\right\vert ^{2}>0$
and the norm lies in $\{\pm1\}=\mathbf{Z}^{\times}$ since $u$ is a unit.
Hence, $N_{K/\mathbf{Q}}(u)=1$ and we can continue the above computation with
\begin{align}
&  =1-\sum_{i=1}^{3}\sigma_{i}(u)+\sum_{1\leq i<j\leq3}\sigma_{i}(u)\sigma
_{j}(u)-N_{K/\mathbf{Q}}(u)\nonumber\\
&  =\sum_{i<j}\sigma_{i}(u)\sigma_{j}(u)-\sum_{i=1}^{3}\sigma_{i}(u)\text{.}
\label{ltta1}%
\end{align}
By $\sigma_{1}(u)\left\vert \sigma_{2}(u)\right\vert ^{2}=1$ we have
$\left\vert \sigma_{2}(u)\right\vert =\sqrt{1/\sigma_{1}(u)}$. Thus, for
$z:=\max(\sigma_{1}(u),\sqrt{1/\sigma_{1}(u)})>0$ we get the estimate%
\[
\#(\mathcal{O}_{K}/J(\mathcal{O}_{K}^{\times,+}))\leq3z^{2}+3z=3(z+z^{2})
\]
from Equation \ref{ltta1}. From Prop. \ref{prop_main} we know that%
\[
\operatorname*{Vol}\left(  X\right)  =\frac{1}{4}\cdot\sqrt{\left\vert
\triangle_{K/\mathbf{Q}}\right\vert }\cdot R_{K}=\frac{1}{4}\cdot
\sqrt{\left\vert \triangle_{K/\mathbf{Q}}\right\vert }\cdot\log\left\vert
\sigma_{1}u\right\vert \text{,}%
\]
since the Dirichlet regulator is just formed from a $(1\times1)$-matrix in the
present situation, and the single entry comes from the logarithmic embedding
of the fundamental unit. Thus, reversing the usual logic, we can also say that%
\[
\sigma_{1}u=\exp\left(  4\frac{\operatorname*{Vol}\left(  X\right)  }%
{\sqrt{\left\vert \triangle_{K/\mathbf{Q}}\right\vert }}\right)  \text{.}%
\]
The claim follows from connecting this with our previous upper bound.
\end{proof}

\begin{example}
\label{Example_VolumeBounds}As usual in this text, let us compare this
estimate to precise values: We shall study the number fields%
\[
K:=\mathbf{Q}[T]/(T^{3}+8T-m)
\]
for $m\geq1$, whenever the given polynomial is irreducible. It is easy to see
that these are number fields with $s=t=1$. The discriminant of the order
$\mathbf{Z}[\overline{T}]\subseteq\mathcal{O}_{K}$ is easily computed to be%
\[
\triangle_{\mathbf{Z}[\overline{T}]/\mathbf{Q}}=-27m^{2}-2048
\]
and for most of the $1\leq m\leq10$ the order $\mathbf{Z}[\overline{T}]$ is
the entire ring of integers or at least has only a small index. With the help
the computer we obtain:%
\[%
\begin{tabular}
[c]{c|c|c|c}%
$m$ & $H_{1}(X,\mathbf{Z})_{\operatorname*{tor}}$ & upper bound &
$\operatorname*{Vol}\left(  X\right)  $\\\hline
\multicolumn{1}{r|}{$1$} & \multicolumn{1}{|r|}{$4$} &
\multicolumn{1}{|r|}{$13.54$} & \multicolumn{1}{|r}{$2.3702...$}\\
\multicolumn{1}{r|}{$2$} & \multicolumn{1}{|r|}{$2$} &
\multicolumn{1}{|r|}{$9.58$} & \multicolumn{1}{|r}{$1.0105...$}\\
\multicolumn{1}{r|}{$3$} & \multicolumn{1}{|r|}{$2856582$} &
\multicolumn{1}{|r|}{$8575220$} & \multicolumn{1}{|r}{$177.8782...$}\\
\multicolumn{1}{r|}{$4$} & \multicolumn{1}{|r|}{$32$} &
\multicolumn{1}{|r|}{$122.47$} & \multicolumn{1}{|r}{$22.1167...$}\\
\multicolumn{1}{r|}{$5$} & \multicolumn{1}{|r|}{$5146$} &
\multicolumn{1}{|r|}{$15731.73$} & \multicolumn{1}{|r}{$111.5530...$}\\
\multicolumn{1}{r|}{$6$} & \multicolumn{1}{|r|}{$288$} &
\multicolumn{1}{|r|}{$1022.58$} & \multicolumn{1}{|r}{$79.3724...$}\\
\multicolumn{1}{r|}{$7$} & \multicolumn{1}{|r|}{$1288$} &
\multicolumn{1}{|r|}{$4175.28$} & \multicolumn{1}{|r}{$104.6757...$}\\
\multicolumn{1}{r|}{$8$} & \multicolumn{1}{|r|}{$2$} &
\multicolumn{1}{|r|}{$11.07$} & \multicolumn{1}{|r}{$1.5189...$}\\
\multicolumn{1}{r|}{$10$} & \multicolumn{1}{|r|}{$14$} &
\multicolumn{1}{|r|}{$43.89$} & \multicolumn{1}{|r}{$41.7309...$}%
\end{tabular}
\]
The values in the two right-hand side columns have been truncated. The
particularly large values for $m=3,5$ are mostly caused by the fact that these
number fields have exceptionally large Dirichlet regulators. Allow me to
emphasize once more that the computation of the upper bounds requires the
determination of the fundamental unit just as does finding the torsion group.
Therefore, this estimate is truly not of any algorithmic use.
\end{example}

\section{\label{sect_SmallestVolume}Smallest volume}

Firstly, we must ask: Is this question well-defined at all?\medskip

Usually, when one looks at questions like%
\[
(\text{complex surfaces})\cap(\text{LCK manifolds})
\]
as in Vaisman's paper \cite{MR1038005}\footnote{This paper seems to have been
written in response to Wall's study \cite{MR827276}, \cite{MR837617}. Taking
inspiration from Thurston's geometries, Wall asks which $4$-dimensional
geometries (= nice simply connected Riemannian real manifolds whose isometry
group acts transitively and admits lattices) possess a complex structure so
that the isometry action is holomorphic. He finds that a complex structure
often exists, often unique, but not always K\"{a}hler.}, or%
\[
\left(  \text{real solvmanifolds}\right)  \cap\left(  \text{LCK manifolds}%
\right)
\]
as suggested in work of Hasegawa \cite{MR2235860}, we might primarily be
interested in the existence of a K\"{a}hler or LCK metric at all. Once such
exists, there can be many, at the very least we can rescale it
(\textquotedblleft K\"{a}hler cones\textquotedblright). In this sense the
volume depends on choices and it is a pointless task to find a smallest volume
among arbitrary choices. However, the situation is a little different for
Oeljeklaus-Toma manifolds.\medskip\newline For finite volume hyperbolic
$n$-manifolds $X,X^{\prime}$ (with $n\geq3$) if there exists an isomorphism of
fundamental groups $\pi_{1}(X,\ast)\overset{\sim}{\longrightarrow}\pi
_{1}(X^{\prime},\ast)$, then there even exists an isometry $\phi
:X\overset{\sim}{\longrightarrow}X^{\prime}$ (Mostow-Prasad Rigidity). In
particular, the volume is a topological invariant; homeomorphic spaces must
have the same volume. This makes it very interesting to study the possible
volumes, and to search for a smallest volume.\newline For the Oeljeklaus-Toma
manifolds $X(K;\mathcal{O}_{K}^{\times,+})$ the Proposition
\ref{prop_reconstruct} creates a somewhat similar situation. We get a
well-defined function%
\[
\operatorname*{Vol}:\left\{
\begin{array}
[c]{c}%
\text{spaces }X\text{ homeomorphic to an}\\
\text{Oeljeklaus-Toma manifold}%
\end{array}
\right\}  \longrightarrow\mathbf{R}%
\]
by associating to any $X$ its \textquotedblleft canonical
model\textquotedblright\ $(\mathbf{H}^{s}\times\mathbf{C})/(\pi_{1}(X,\ast))$,
which comes with the standard normalized Oeljeklaus-Toma metric. So at least
after fixing once and for all a normalized metric (as we have done in this
text), we get a well-defined volume and in particular well-defined infimum of volumes.

The situation might be quite different for the spaces $X(K;U)$ for $t>1$
complex places. As Example \ref{Example_CannotReconstructForMoreComplexPlaces}
shows, there are different number fields $K,K^{\prime}$ and admissible
subgroups $U,U^{\prime}$ so that there exists a diffeomorphism%
\[
\frac{\mathbf{H}^{s}\times\mathbf{C}^{t}}{\mathcal{O}_{K}\rtimes U}%
\overset{\sim}{\longrightarrow}\frac{\mathbf{H}^{s}\times\mathbf{C}^{t}%
}{\mathcal{O}_{K^{\prime}}\rtimes U^{\prime}}\text{,}%
\]
yet even if there happens to exist a normalized LCK metric (for example
Battisti's generalized Oeljeklaus-Toma metric, \cite[Appendix]{MR3193953}), I
would suspect the volumes to differ. Example
\ref{Example_CannotReconstructForMoreComplexPlaces}\ however says nothing
about this since these spaces do not admit any LCK\ metrics for sure, as we
explain \textit{loc. cit}.\medskip

This being said and an overall normalization chosen, let us investigate
whether there is a smallest volume. Certainly, the infimum of volumes could
just be zero. For those readers who like the bridge to hyperbolic
$3$-manifolds as alluded to in \S \ref{sect_toymodelproducthyperbolic}, it
should be said that there is a unique smallest compact orientable hyperbolic
$3$-manifold, the Weeks manifold \cite{MR1882023}, \cite{MR2525782}. Its
volume is%
\[
\frac{3\cdot23^{\frac{3}{2}}}{4\pi^{4}}\zeta_{K}(2)\qquad\text{for}\qquad
K:=\mathbf{Q}[T]/(T^{3}-T+1)\text{.}%
\]
This cubic number field $K$ is the one whose discriminant has the smallest
absolute value among all cubic fields. The volume of its Oeljeklaus-Toma
manifold is $\approx0.33714644$. Surprisingly, it turns out that this is also
the smallest possible volume of an Oeljeklaus-Toma manifold with $s=1$.

By quoting some rather hard results from analytic number theory and the
geometry of numbers, one can show with little effort that, once fixing a
number of real places $s$, the volume among all Oeljeklaus-Toma manifolds
generally stays bounded away from zero:

\begin{proposition}
\label{prop_intext_minvol}For every $s\geq1$ there exists a unique real number
$\operatorname*{Vol}\nolimits_{s}$ so that the following holds:

\begin{enumerate}
\item All Oeljeklaus-Toma manifolds with fixed $s$ have volume $\geq
\operatorname*{Vol}\nolimits_{s}$.

\item There exists at least one, but at most finitely many, actually attaining
this minimal volume $\operatorname*{Vol}\nolimits_{s}$.

\item We have the crude lower bound
\[
\operatorname*{Vol}\nolimits_{s}\geq\pi\frac{(s+2)^{s+1}}{4^{s+2}\cdot
2^{s^{2}}\cdot s!}\text{.}%
\]

\end{enumerate}

For the special case $s=1$ there is a unique Oeljeklaus-Toma manifold of
smallest volume, namely%
\[
\operatorname*{Vol}\nolimits_{1}=0.337146\ldots
\]
It is the one coming from the number field%
\[
K:=\mathbf{Q}[T]/(T^{3}-T+1)\text{.}%
\]

\end{proposition}

\begin{proof}
All the real work here lies in a deep result of Friedman \cite{MR1022309},
based on earlier work of Remak and Zimmert. We have:

\begin{itemize}
\item For every number field $K$, apart from three exceptions with
$[K:\mathbf{Q}]=6$, we have $R_{K}>\frac{1}{4}$ (\cite[Theorem B]{MR1022309}).

\item For every number field $K$ with $s=t=1$ and $\left\vert \triangle
_{K/\mathbf{Q}}\right\vert <18.7^{3}$ we have $R_{K}/2\geq0.14$ (\cite[Prop.
2.2 and Table 2 for $(r_{1},r_{2})=(1,1)$]{MR1022309}).
\end{itemize}

If one is willing to accept far weaker bounds, a short proof of a lower bound
for the regulator in terms of $s$ can also be found in \cite{MR1225260}. Let
$X$ be an arbitrary Oeljeklaus-Toma manifold for a given $s\geq1$. From the
first estimate and Prop. \ref{prop_main} we readily obtain the bound%
\[
\operatorname*{Vol}\left(  X\right)  >\frac{(s+1)}{4^{s+1}\cdot2^{s^{2}}}%
\cdot\sqrt{\left\vert \triangle_{K/\mathbf{Q}}\right\vert }\text{,}%
\]
except for finitely many fields and we can ignore them as this does not affect
the validity of our claim (since their regulators are explicitly known and
listed in Friedman's work, we could also just work with the overall minimal
regulator). Furthermore, there is the standard Minkowski discriminant estimate%
\[
\sqrt{\left\vert \triangle_{K/\mathbf{Q}}\right\vert }\geq\left(  \frac{\pi
}{4}\right)  \frac{n^{n}}{n!}\qquad\text{for }n:=s+2
\]
the degree of the field. Combining these inequalities, we arrive at%
\begin{equation}
\operatorname*{Vol}\left(  X\right)  >\pi\frac{(s+1)(s+2)^{s+2}}{4^{s+2}%
\cdot2^{s^{2}}(s+2)!}\text{.} \label{lja1}%
\end{equation}
Work of Odlyzko, Martinet, and many others would give much better lower bounds
for particular ranges of $s$, but this estimate suffices for our needs. Define%
\[
V_{s}=\{\operatorname*{Vol}\left(  X\right)  \mid X(K,\mathcal{O}_{K}%
^{\times,+})\text{ for any }K\text{ with }t=1\text{ and given }s\}\subset
\mathbf{R}\text{,}%
\]
the set of all volumes that can occur for fixed $s$. This set is non-empty and
bounded from below by Equation \ref{lja1}, so it will have some infimum
$\wp:=\inf(V_{s})$. We now argue by contradiction: Suppose there is no $X$
whose volume attains this infimum. This means that there exists a sequence of
number fields $K_{n}$ so that%
\begin{align}
\wp &  =\underset{n\rightarrow\infty}{\lim}\operatorname*{Vol}\left(
X(K_{n},\mathcal{O}_{K_{n}}^{\times,+})\right) \nonumber\\
&  =\frac{(s+1)}{4^{s}\cdot2^{s^{2}}}\cdot\underset{n\rightarrow\infty}{\lim
}\sqrt{\left\vert \triangle_{K_{n}/\mathbf{Q}}\right\vert }\cdot R_{K_{n}%
}\text{.} \label{lja2}%
\end{align}
The Hermite-Minkowski Theorem tells us that there are only finitely many
number fields of bounded discriminant $\left\vert \triangle_{K/\mathbf{Q}%
}\right\vert <C$ for any $C\geq0$, so if the sequence $(\left\vert
\triangle_{K_{n}/\mathbf{Q}}\right\vert )_{n\geq0}$ stays bounded,
$\{K_{0},K_{1},K_{2},\ldots\}$ is actually a finite set and therefore some
$K_{i}$ will realize the infimum, contradicting our assumption. Thus, we must
have $\lim\nolimits_{n\rightarrow\infty}\sqrt{\left\vert \triangle
_{K_{n}/\mathbf{Q}}\right\vert }=+\infty$. Hence, from Equation \ref{lja2} we
can deduce that $\lim\nolimits_{n\rightarrow\infty}R_{K_{n}}=0$. This
contradicts Friedman's bound $R_{K_{n}}>\frac{1}{4}$. Thus, there exists at
least one $K_{i}$ with $\operatorname*{Vol}\left(  X(K_{n},\mathcal{O}_{K_{n}%
}^{\times,+})\right)  =\wp$. If $\{K_{i}\}$ now denotes the (possibly
infinite) set of all number fields realizing the volume $\wp$, that is%
\[
\wp=\frac{(s+1)}{4^{s}\cdot2^{s^{2}}}\cdot\underset{n\rightarrow\infty}{\lim
}\sqrt{\left\vert \triangle_{K_{n}/\mathbf{Q}}\right\vert }\cdot R_{K_{n}%
}\text{,}%
\]
\newline the same argument as above shows that the set $\{K_{i}\}$ must be
finite, for otherwise the discriminants grow arbitrarily large, ultimately
forcing regulators $\leq\frac{1}{4}$, which is impossible.

Next, consider the case $s=t=1$: Firstly, (for $\left\vert \triangle
_{K/\mathbf{Q}}\right\vert \geq18.7^{3}$) the first Friedman estimate shows
that%
\[
\operatorname*{Vol}\left(  X\right)  =\frac{1}{4}\cdot\sqrt{\left\vert
\triangle_{K/\mathbf{Q}}\right\vert }\cdot R_{K}>\frac{1}{16}\cdot
\sqrt{18.7^{3}}>5\text{.}%
\]
Next, suppose $\left\vert \triangle_{K/\mathbf{Q}}\right\vert <18.7^{3}$. Then
the second Friedman estimate implies%
\[
\operatorname*{Vol}\left(  X\right)  =\frac{1}{4}\cdot\sqrt{\left\vert
\triangle_{K/\mathbf{Q}}\right\vert }\cdot R_{K}>\frac{0.28}{4}\cdot
\sqrt{\left\vert \triangle_{K/\mathbf{Q}}\right\vert }%
\]
and since the smallest possible discriminant of a cubic field is $\left\vert
\triangle_{K/\mathbf{Q}}\right\vert =23$, we deduce $\operatorname*{Vol}%
\left(  X\right)  >0.335708$. Thus, we have a good lower bound for the
smallest possible volume. Next, let us assume that $K$ has a discriminant of
absolute value larger than $23$, so at least $24$. Then Friedman's bound shows
that%
\[
\operatorname*{Vol}\left(  X\right)  >\frac{0.28}{4}\cdot\sqrt{24}%
>0.3429\text{.}%
\]
Since the Oeljeklaus-Toma manifold of $K:=\mathbf{Q}[T]/(T^{3}-T+1)$ has the
underlined volume in%
\[
0.335708<\underline{0.3371\ldots}<0.3429\text{,}%
\]
we deduce that the minimal volume can (and is) attained only for number fields
$K$ with $s=t=1$ and discriminant $\left\vert \triangle_{K/\mathbf{Q}%
}\right\vert =23$. However, in the present case it is known that there exists
only one number field with discriminant of absolute value $23$.
\end{proof}

\begin{proof}
[Proof of Prop. \ref{prop_boundedminvolume}]This is just a reformulation of
the previous result, using that the dimension of an Oeljeklaus-Toma manifold
is $\dim X=s+2$.
\end{proof}

After the Weeks manifold, the compact oriented arithmetic hyperbolic
$3$-manifold of next larger volume is the Meyerhoff manifold, \cite{MR1882023}%
. It was shown by Chinburg \cite{MR883417} to be arithmetic and to have volume%
\[
\frac{12\cdot283^{\frac{3}{2}}}{(2\pi)^{6}}\zeta_{K}(2)\qquad\text{for}\qquad
K:=\mathbf{Q}[T]/(T^{4}-T-1)\text{.}%
\]
This quartic number field $K$ has $s=2$ and $t=1$ real resp. complex places
and discriminant $-283$. It is known that the smallest discriminants for these
numbers of places are as given on the left-hand side column in the following
table:%
\[%
\begin{tabular}
[c]{c|c|c}%
$\triangle_{K/\mathbf{Q}}$ & $\operatorname*{Vol}\left(  X\right)  $ & min.
polynomial\\\hline
\multicolumn{1}{r|}{$-275$} & \multicolumn{1}{|r|}{$0.0717$} &
\multicolumn{1}{|r}{$T^{4}-T^{3}+2T-1$}\\
\multicolumn{1}{r|}{$-283$} & \multicolumn{1}{|r|}{$0.0745$} &
\multicolumn{1}{|r}{$T^{4}-T-1$}\\
\multicolumn{1}{r|}{$-331$} & \multicolumn{1}{|r|}{$0.0921$} &
\multicolumn{1}{|r}{$T^{4}-T^{3}+T^{2}+T-1$}\\
\multicolumn{1}{r|}{$-400$} & \multicolumn{1}{|r|}{$0.1196$} &
\multicolumn{1}{|r}{$T^{4}-T^{2}-1$}\\
\multicolumn{1}{r|}{$-475$} & \multicolumn{1}{|r|}{$0.1473$} &
\multicolumn{1}{|r}{$T^{4}-2T^{3}+T^{2}-2T+1$}%
\end{tabular}
\]
We leave it to the reader to show that the middle column indeed gives the
smallest four possible volumes for Oeljeklaus-Toma manifolds with two real
places. One can proceed as in the argument above, this time using Friedman's
estimate $R_{K}/2>0.1835$ for $\left\vert \triangle_{K/\mathbf{Q}}\right\vert
\leq36^{4}$, \cite[Prop. 2.2 and Table 2 for $(r_{1},r_{2})=(2,1)$]{MR1022309}.

\begin{example}
We will now determine the minimal volumes of Oeljeklaus-Toma manifolds for
$s=1,2,3,4,5$. We follow the same method as in the proof of Prop.
\ref{prop_intext_minvol}, but suppress a number of details and just explain
the general pattern. It would seem entirely hopeless to me to perform the
necessary verifications below without the help of a computer. Using tables for
minimal known discriminants, we first compile the following table:%
\begin{equation}%
\begin{tabular}
[c]{c|c|c|c|c|c}%
$s$ & $R_{K}^{>}$ & $\left\vert \triangle_{K/\mathbf{Q}}\right\vert _{1^{st}}$
& $\left\vert \triangle_{K/\mathbf{Q}}\right\vert _{2^{nd}}$ &
$\operatorname*{Vol}$ of $1^{st}$ & $V_{2^{nd}}^{>}$\\\hline
\multicolumn{1}{r|}{$1$} & \multicolumn{1}{|r|}{$0.28$} &
\multicolumn{1}{|r|}{$23$} & \multicolumn{1}{|r|}{$31$} &
\multicolumn{1}{|r|}{$0.33714$} & \multicolumn{1}{|r}{$0.38974$}\\
\multicolumn{1}{r|}{$2$} & \multicolumn{1}{|r|}{$0.367$} &
\multicolumn{1}{|r|}{$275$} & \multicolumn{1}{|r|}{$283$} &
\multicolumn{1}{|r|}{$0.07174$} & \multicolumn{1}{|r}{$0.07235$}\\
\multicolumn{1}{r|}{$3$} & \multicolumn{1}{|r|}{$0.6218$} &
\multicolumn{1}{|r|}{$4511$} & \multicolumn{1}{|r|}{$4903$} &
\multicolumn{1}{|r|}{$0.00515$} & \multicolumn{1}{|r}{$0.00531$}\\
\multicolumn{1}{r|}{$4$} & \multicolumn{1}{|r|}{$1.2376$} &
\multicolumn{1}{|r|}{$92779$} & \multicolumn{1}{|r|}{$94363$} &
\multicolumn{1}{|r|}{$0.0001146$} & \multicolumn{1}{|r}{$0.0001133$}\\
\multicolumn{1}{r|}{$5$} & \multicolumn{1}{|r|}{$2.7822$} &
\multicolumn{1}{|r|}{$2306599$} & \multicolumn{1}{|r|}{$2369207$} &
\multicolumn{1}{|r|}{$7.650\cdot10^{-7}$} & \multicolumn{1}{|r}{$7.478\cdot
10^{-7}$}%
\end{tabular}
\label{lTable1}%
\end{equation}
Here the column \textquotedblleft$R_{K}^{>}$\textquotedblright\ lists a lower
bound for the regulator of all number fields with given $s$ and $t=1$. We just
copied these values from the work of Friedman (\cite[Table $2$ for
$(r_{1},r_{2})=(s,1)$]{MR1022309}), noting that his table spells out lower
bounds for $R_{K}/2$. His values are only valid for discriminants smaller than
certain bounds also given in \cite[Table $2$]{MR1022309}, but these are
harmless in all cases we deal with. Unsurprisingly so, as we are mostly
interested in the smallest possible discriminants. The columns
\textquotedblleft$\left\vert \triangle_{K/\mathbf{Q}}\right\vert _{1^{st}}%
$\textquotedblright\ and \textquotedblleft$\left\vert \triangle_{K/\mathbf{Q}%
}\right\vert _{2^{nd}}$\textquotedblright\ list the smallest and second
smallest discriminant possible for the given $s$ and $t=1$. In principle there
could be several number fields realizing the smallest discriminant, but in all
cases we touch here, there is a unique one:
\begin{equation}%
\begin{tabular}
[c]{c|c}%
$s$ & number field of $\left\vert \triangle_{K/\mathbf{Q}}\right\vert
_{1^{st}}$\\\hline
\multicolumn{1}{r|}{$1$} & \multicolumn{1}{|r}{$T^{3}-T^{2}+1$}\\
\multicolumn{1}{r|}{$2$} & \multicolumn{1}{|r}{$T^{4}-T^{3}+2T-1$}\\
\multicolumn{1}{r|}{$3$} & \multicolumn{1}{|r}{$T^{5}-T^{3}-2T^{2}+1$}\\
\multicolumn{1}{r|}{$4$} & \multicolumn{1}{|r}{$T^{6}-T^{5}-2T^{4}%
+3T^{3}-T^{2}-2T+1$}\\
\multicolumn{1}{r|}{$5$} & \multicolumn{1}{|r}{$T^{7}-3T^{5}-T^{4}%
+T^{3}+3T^{2}+T-1$}%
\end{tabular}
\label{lTable2}%
\end{equation}
We compute their regulators with the help of a computer and therefore obtain
the volumes of the Oeljeklaus-Toma manifolds associated to the number field of
smallest possible discriminant. These values are listed in the column
\textquotedblleft$\operatorname*{Vol}$ of $1^{st}$\textquotedblright. Next, we
use Friedman's bound to compute a lower bound on the volumes of all
Oeljeklaus-Toma manifolds from number fields of discriminant at least the
second smallest, i.e. in the column \textquotedblleft$V_{2^{nd}}^{>}%
$\textquotedblright\ we list%
\begin{equation}
\frac{(s+1)}{4^{s}\cdot2^{s^{2}}}\sqrt{\left\vert \triangle_{K/\mathbf{Q}%
}\right\vert _{2^{nd}}}\cdot\text{(Friedman~bound }R_{K}^{>}\text{).}
\label{lja3}%
\end{equation}
The cases $s=1,2,3$ are obvious now: We find that as soon as we use number
fields whose discriminants are second smallest or larger, we will exceed the
volume of the Oeljeklaus-Toma manifold made from the number field of smallest
discriminant. The case $s=4$ is more involved since we see that the
Oeljeklaus-Toma manifold of the unique sextic field of smallest discriminant
has a volume strictly larger than a volume that could hypothetically occur for
the second smallest discriminant as well. In fact, the second smallest
discriminant for $s=4$, that is $-94363$, is realized by the number field of%
\[
T^{6}-2T^{4}-2T^{3}+3T+1\text{.}%
\]
We compute its volume to be $0.000116$, so it is not smaller. The next larger
discriminant is known to be $\left\vert \triangle_{K/\mathbf{Q}}\right\vert
_{3^{rd}}=103243$, and Friedman's bound as in Equation \ref{lja3} yields a
minimal volume of $0.00011851$ for discriminants $\geq\left\vert
\triangle_{K/\mathbf{Q}}\right\vert _{3^{rd}}$. This settles the case: Still,
the manifold coming from the number field of smallest discriminant has also
the smallest volume. For $s=5$ the same happens. The second smallest
discriminant is realized by%
\begin{equation}
T^{7}-4T^{5}+3T^{3}-T^{2}+T+1 \label{lja4}%
\end{equation}
and we compute its volume to be $7.88\cdot10^{-7}$, so its volume is larger.
We have $\left\vert \triangle_{K/\mathbf{Q}}\right\vert _{3^{rd}}=2616839$ and
Friedman's bound shows that for this and larger discriminants the volumes must
be at least $7.85\cdot10^{-7}$. This confirms that we have found the smallest
one, but do not know for sure whether Equation \ref{lja4} defines the second
smallest one.\newline We conclude: The values listed under \textquotedblleft%
$\operatorname*{Vol}$ of $1^{st}$\textquotedblright\ in Table \ref{lTable1}
provide the smallest possible volume for the given $s$, and in each case are
realized by only a single manifold; and these are the ones listed in Table
\ref{lTable2}.
\end{example}

Although all of the above computations might suggest that the volumes follow
the ordering of ascending discriminants, this simple pattern completely
collapses as we get farther from the minimal volumes. The polynomials in the
following table have been chosen rather at random, but ordering the rows by
increasing volume shows that this does not imply much about the ordering of
the discriminants:%
\[%
\begin{tabular}
[c]{c|c|c}%
$\triangle_{K/\mathbf{Q}}$ & $\operatorname*{Vol}\left(  X\right)  $ & min.
polynomial\\\hline
\multicolumn{1}{r|}{$-1931$} & \multicolumn{1}{|r|}{$0.7162$} &
\multicolumn{1}{|r}{$T^{4}+3T+1$}\\
\multicolumn{1}{r|}{$-6371$} & \multicolumn{1}{|r|}{$3.0870$} &
\multicolumn{1}{|r}{$T^{4}+13T+1$}\\
\multicolumn{1}{r|}{$-8123$} & \multicolumn{1}{|r|}{$3.5939$} &
\multicolumn{1}{|r}{$T^{4}-4T^{3}-T-1$}\\
\multicolumn{1}{r|}{$-12675$} & \multicolumn{1}{|r|}{$4.6792$} &
\multicolumn{1}{|r}{$T^{4}-8T^{3}-T-1$}\\
\multicolumn{1}{r|}{$-6656$} & \multicolumn{1}{|r|}{$5.3600$} &
\multicolumn{1}{|r}{$T^{4}-4T+1$}\\
\multicolumn{1}{r|}{$-16619$} & \multicolumn{1}{|r|}{$7.5061$} &
\multicolumn{1}{|r}{$T^{4}-5T+1$}\\
\multicolumn{1}{r|}{$-8684$} & \multicolumn{1}{|r|}{$9.2152$} &
\multicolumn{1}{|r}{$T^{4}-6T+1$}%
\end{tabular}
\]

Although Friedman's estimates are very non-trivial results, the result that
the Weeks manifold has smallest volume among compact hyperbolic $3$-manifolds
is of a completely different level of complexity. This remark truly applies to
any comparison we make between Oeljeklaus-Toma manifolds and hyperbolic or
product-hyperbolic geometries in this text.

\begin{acknowledgement}
I would like to express my sincere gratitude to Victor Vuletescu for teaching
me a lot of things, not all of them of mathematical nature. This note is a
direct result of his inspiring ideas about the interplay of geometric and
arithmetic conditions in Oeljeklaus-Toma manifolds. I also thank Chris
Wuthrich for introducing me to SAGE.
\end{acknowledgement}

\bibliographystyle{alpha}
\bibliography{ollinewbib}

\newpage

\section{Appendix}

\begin{code}
\label{comp_code_SageForExampleNoReconstruct}The computations underlying
Example \ref{Example_CannotReconstructForMoreComplexPlaces} can be confirmed
in an automated fashion by computer algebra systems. The following code is
written for \textrm{SAGE} \cite{sage}, largely using \textrm{PARI/GP}
\cite{PARI2}. Firstly, we confirm that $S$ was a generator of the group of
units (up to torsion):\medskip\newline\texttt{L1.%
%TCIMACRO{\TEXTsymbol{<}}%
%BeginExpansion
$<$%
%EndExpansion
s%
%TCIMACRO{\TEXTsymbol{>} }%
%BeginExpansion
$>$
%EndExpansion
= NumberField(x\symbol{94}3+x+1)\newline print L1.unit\_group().gens()\medskip
}\newline This computation can also be done by hand using the Minkowski
bounds; but remember that this verification was actually not needed for the
validity of the example. Next, we check the crucial fact that the order
$J_{i}$ is maximal, i.e. that it is the ring of integers:\medskip
\newline\texttt{from sage.rings.number\_field.order import *\newline E.%
%TCIMACRO{\TEXTsymbol{<}}%
%BeginExpansion
$<$%
%EndExpansion
s,t%
%TCIMACRO{\TEXTsymbol{>} }%
%BeginExpansion
$>$
%EndExpansion
= NumberField([x\symbol{94}3+x+1,x\symbol{94}3-x+2])\newline C.%
%TCIMACRO{\TEXTsymbol{<}}%
%BeginExpansion
$<$%
%EndExpansion
w%
%TCIMACRO{\TEXTsymbol{>} }%
%BeginExpansion
$>$
%EndExpansion
= E.absolute\_field()\newline V, from\_v, to\_v = C.vector\_space()\newline J
= span([to\_v(1), to\_v(s), to\_v(s\symbol{94}2), to\_v(t), to\_v(t\symbol{94}%
2),}\newline\texttt{to\_v(s*t), to\_v(s\symbol{94}2*t), to\_v(s*t\symbol{94}%
2), to\_v(s\symbol{94}2*t\symbol{94}2)],ZZ)\newline O = AbsoluteOrder(C,
J)\newline print O.is\_maximal()\medskip}\newline Adapt the minimal polynomial
for $t$ to check both cases $i=2,3$. Finally, we check that the compositum of
the Galois closures has degree 216:\medskip\newline\texttt{L1.%
%TCIMACRO{\TEXTsymbol{<}}%
%BeginExpansion
$<$%
%EndExpansion
s%
%TCIMACRO{\TEXTsymbol{>} }%
%BeginExpansion
$>$
%EndExpansion
= NumberField(x\symbol{94}3+x+1)\newline L2.%
%TCIMACRO{\TEXTsymbol{<}}%
%BeginExpansion
$<$%
%EndExpansion
t%
%TCIMACRO{\TEXTsymbol{>} }%
%BeginExpansion
$>$
%EndExpansion
= NumberField(x\symbol{94}3-x+2)\newline L3.%
%TCIMACRO{\TEXTsymbol{<}}%
%BeginExpansion
$<$%
%EndExpansion
u%
%TCIMACRO{\TEXTsymbol{>} }%
%BeginExpansion
$>$
%EndExpansion
= NumberField(x\symbol{94}3-x+1)\newline H1 =
L1.galois\_group(names='b').splitting\_field()\newline H2 =
L2.galois\_group(names='c').splitting\_field()\newline H3 =
L3.galois\_group(names='d').splitting\_field()\newline B =
H1.composite\_fields(H2)[0]\newline C = B.composite\_fields(H3)[0]\newline
print C.degree()\medskip}\newline Of course it would not be particularly hard
to perform this computation by hand, just a bit tedious.
\end{code}

\begin{code}
\label{comp_code_ComputeIdealJ}We discuss the determination of the ideal
$J(U)$, Definition \ref{def_IdealJ}, by computer. We have used this for our
Example \ref{example_ComputeJ2}. The following code runs through the number
fields generated by the minimal polynomials $Z^{3}-Z+h$, whenever these are
irreducible, for $h=1,\ldots,9$. In this particular case these number fields
have $s=t=1$ real resp. complex places, so $\mathcal{O}_{K}^{\times}%
\simeq\left\langle -1\right\rangle \times\mathbf{Z}\left\langle u\right\rangle
$, where $u$ is a fundamental unit. For these minimal polynomials the single
real embedding of the fundamental unit always happens to be negative. This
follows from Descartes' Rule of Signs: The polynomial rewritten in $-Z$ is
$-Z^{3}+Z+h$, which has precisely one sign change among its coefficients.
Therefore, it must have a single negative real root. Hence, $\mathcal{O}%
_{K}^{\times,+}\simeq\mathbf{Z}\left\langle -u\right\rangle $ and the ideal
$J(\mathcal{O}_{K}^{\times,+})$ is generated by the single elemet $1-(-u)=1+u$
by Lemma \ref{Lemma_JUCanBeDefinedOnGenerators}.\medskip

\texttt{z = QQ['z'].0}

\texttt{for h in range(1,10):}

\texttt{\qquad if (z\symbol{94}3-z+h).is\_irreducible():}

\texttt{\qquad\qquad L.%
%TCIMACRO{\TEXTsymbol{<}}%
%BeginExpansion
$<$%
%EndExpansion
s%
%TCIMACRO{\TEXTsymbol{>} }%
%BeginExpansion
$>$
%EndExpansion
= NumberField(z\symbol{94}3-z+h)}

\texttt{\qquad\qquad U = L.unit\_group()}

\texttt{\qquad\qquad T = 1+U.gen(1)}

\texttt{\qquad\qquad J = L.ideal(T);}

\texttt{\qquad\qquad print s.minpoly(), " -%
%TCIMACRO{\TEXTsymbol{>} }%
%BeginExpansion
$>$
%EndExpansion
", factor(J.norm())\medskip}

Note that \textrm{SAGE}\ always returns the unit group in the format so that
\texttt{U.gen(0)} is the torsion generator and \texttt{U.gen(1)} the
non-torsion generator. Hence, in this particular case the ideal $J$ needs to
be generated by \texttt{1+U.gen(1)}. This code can easily be adapted to
similar computations. For example, for the polynomials $Z^{7}-Z-h$ we will
have $s=1$ and $t=3$ real resp. complex places. Any such polynomial has
exactly one sign change in its coefficients, so by Descartes' Rule it has
precisely one positive real root. Hence, $\mathcal{O}_{K}^{\times}%
=\{\pm1\}\times\mathcal{O}_{K}^{\times,+}$ and therefore $J(\mathcal{O}%
_{K}^{\times,+})$ is generated by the elements $1-u$, where $u$ runs through
the generators of $\mathcal{O}_{K}^{\times,+}$, again by Lemma
\ref{Lemma_JUCanBeDefinedOnGenerators}. Replace the definition of \texttt{T}
by\medskip

\texttt{T = [(1-U.gen(i)) for i in range(1,len(U.gens()))]\medskip},

since we now will have several generators. We discard the generator at
\texttt{i=0 }since this is again the torsion generator.
\end{code}

\bigskip

\bigskip
\end{document}